\documentclass[12pt]{article}
\usepackage{amsthm,amsfonts, amsbsy, amssymb,amsmath,graphicx}
\usepackage{graphics}
\usepackage{cite}

\def\thtext#1{
\catcode`@=11
\gdef\@thmcountersep{. #1}
\catcode`@=12}

\long\def\notes#1#2{
\begin{table}[b]
\vspace{#1pt}
\rule{3truecm}{0.3pt}
\vskip4pt
\par\noindent
#2
\end{table}}

\def\thank#1{\parbox{\hsize}{
\par \tolerance=200 \parindent=10pt \small #1.}
\vskip1pt \par\noindent}

\def\subclass#1{\parbox{\hsize}{
\tolerance=200 \parindent=10pt
{\small\it 2010 Mathematical Subject Classification}.
\small #1.}
\vskip1pt \par\noindent}

\def\keywords#1{\parbox{\hsize}{
\tolerance=200 \parindent=10pt
{\it Key words and phrases}. \small #1.} \vskip1pt}

\newtheorem{theorem}{Theorem}[section]

\newtheorem{statement}{Statement}[section]

\newcounter{Rk}[section]
\renewcommand{\thtext}{\thesection.\arabic{Rk}}
\newenvironment{remark}{\trivlist\item[\hskip\labelsep{\bf Remark}]
\refstepcounter{Rk}{\bf\thesection.\arabic{Rk}.}}%
{\endtrivlist}

\newcounter{Df}[section]
\renewcommand{\thtext}{\thesection.\arabic{Df}}
\newenvironment{definition}{\trivlist\item[\hskip\labelsep{\bf Definition}]
\par\refstepcounter{Df}{\bf\thesection.\arabic{Df}.}}%
{\endtrivlist}

\newcounter{Ex}[section]
\renewcommand{\thtext}{\thesection.\arabic{Ex}}
\newenvironment{example}{\trivlist\item[\hskip\labelsep{\bf Example}]
\par\refstepcounter{Ex}{\bf\thesection.\arabic{Ex}.}}%
{\endtrivlist}

\newcommand{\skcrro}{\raisebox{-0.25\height}{\includegraphics[width=0.5cm]{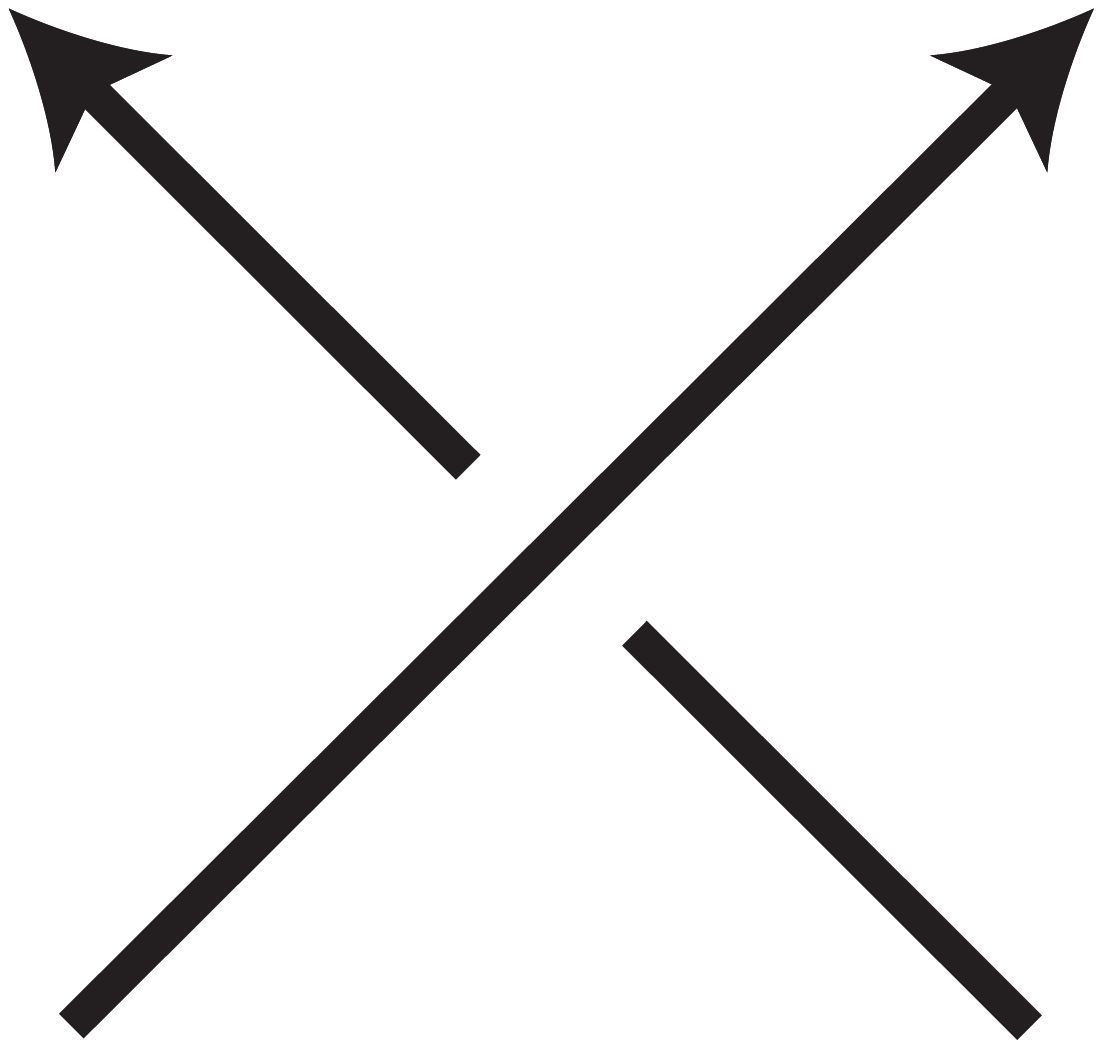}}}
\newcommand{\skcrlo}{\raisebox{-0.25\height}{\includegraphics[width=0.5cm]{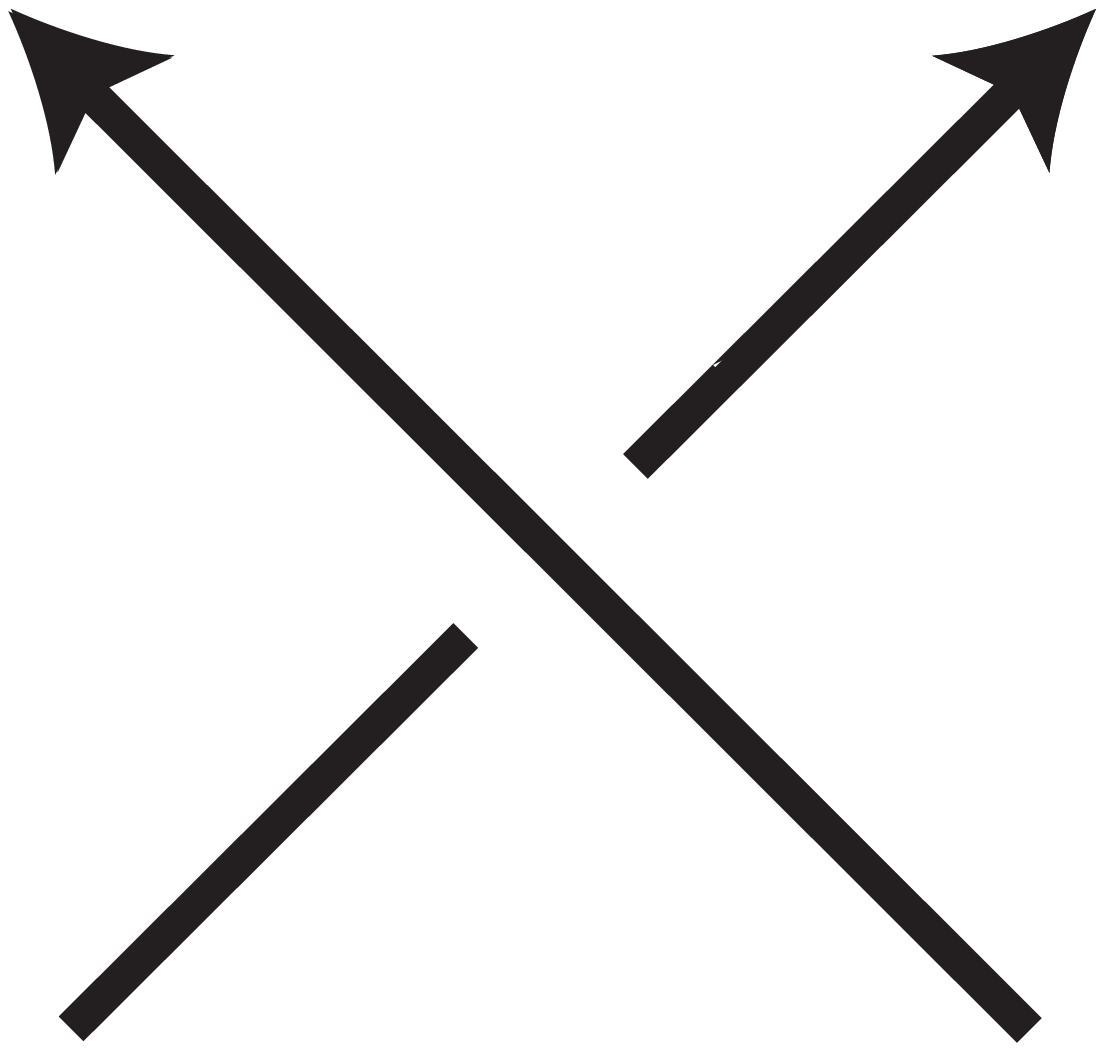}}}
\newcommand{\skcrossr}{\raisebox{-0.25\height}{\includegraphics[width=0.5cm]{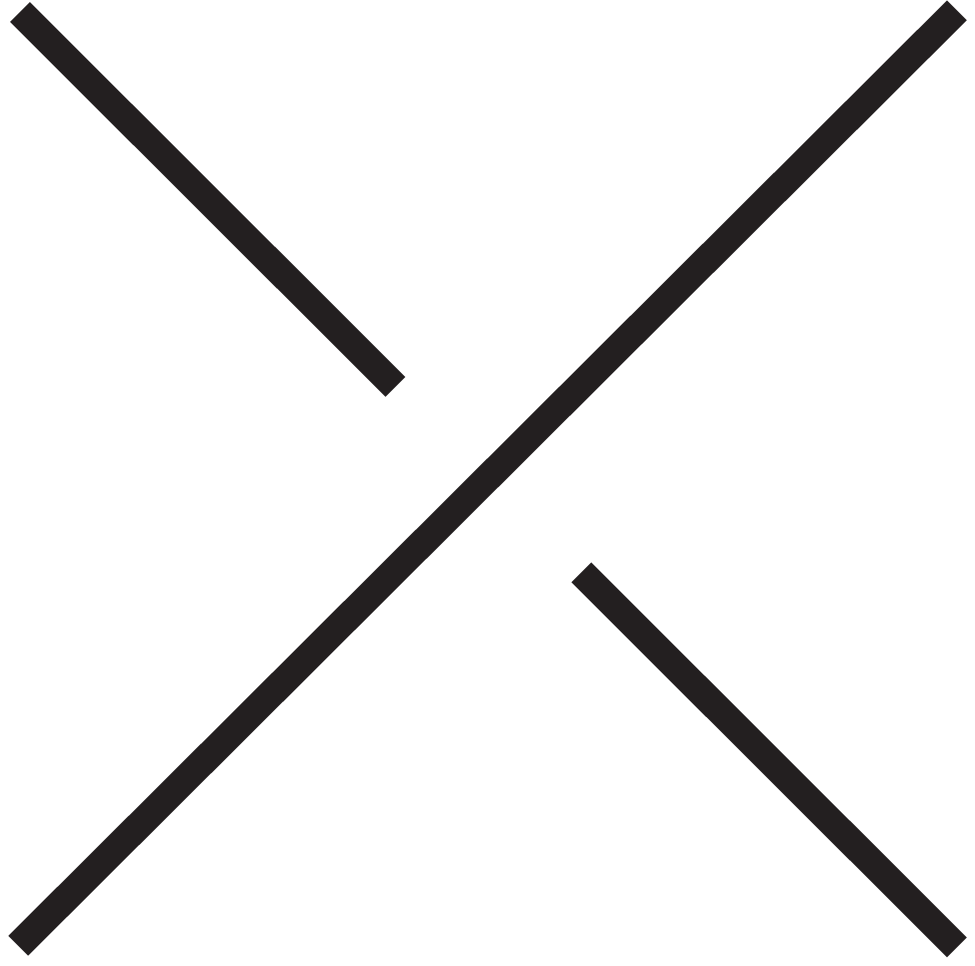}}}
\newcommand{\skcrv}{\raisebox{-0.25\height}{\includegraphics[width=0.5cm]{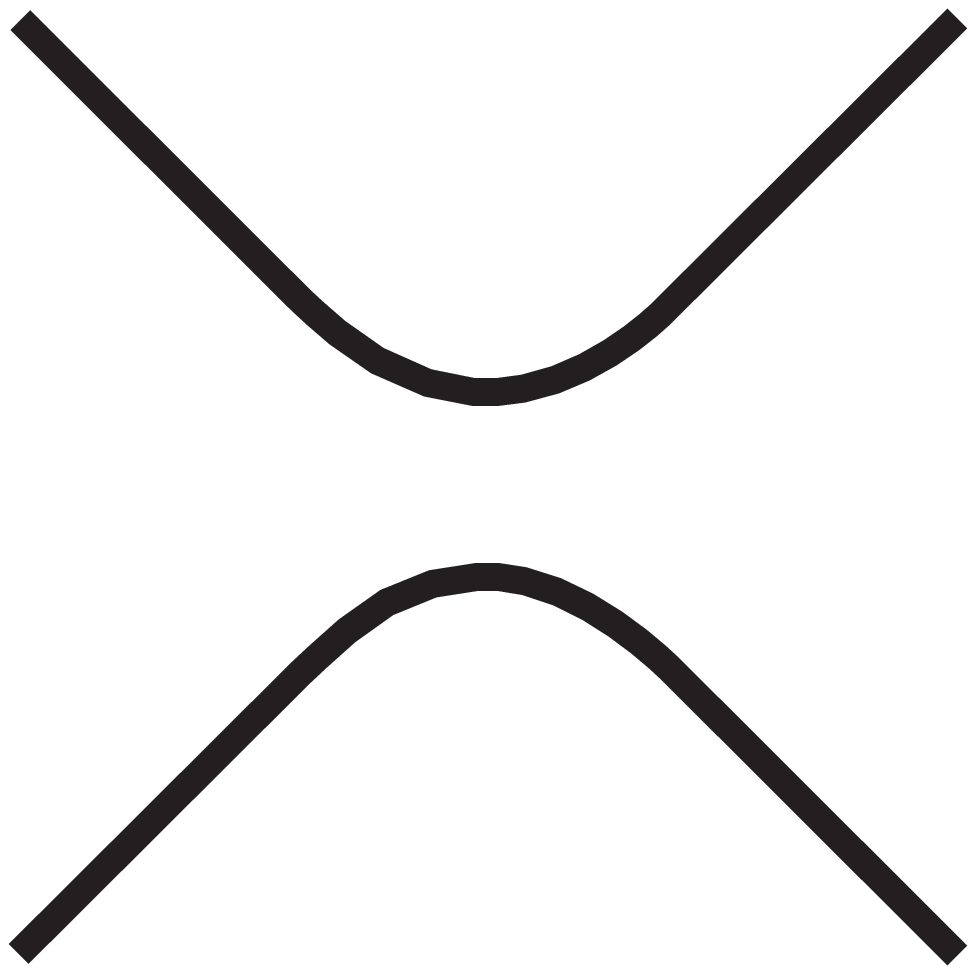}}}
\newcommand{\skcrh}{\raisebox{-0.25\height}{\includegraphics[width=0.5cm]{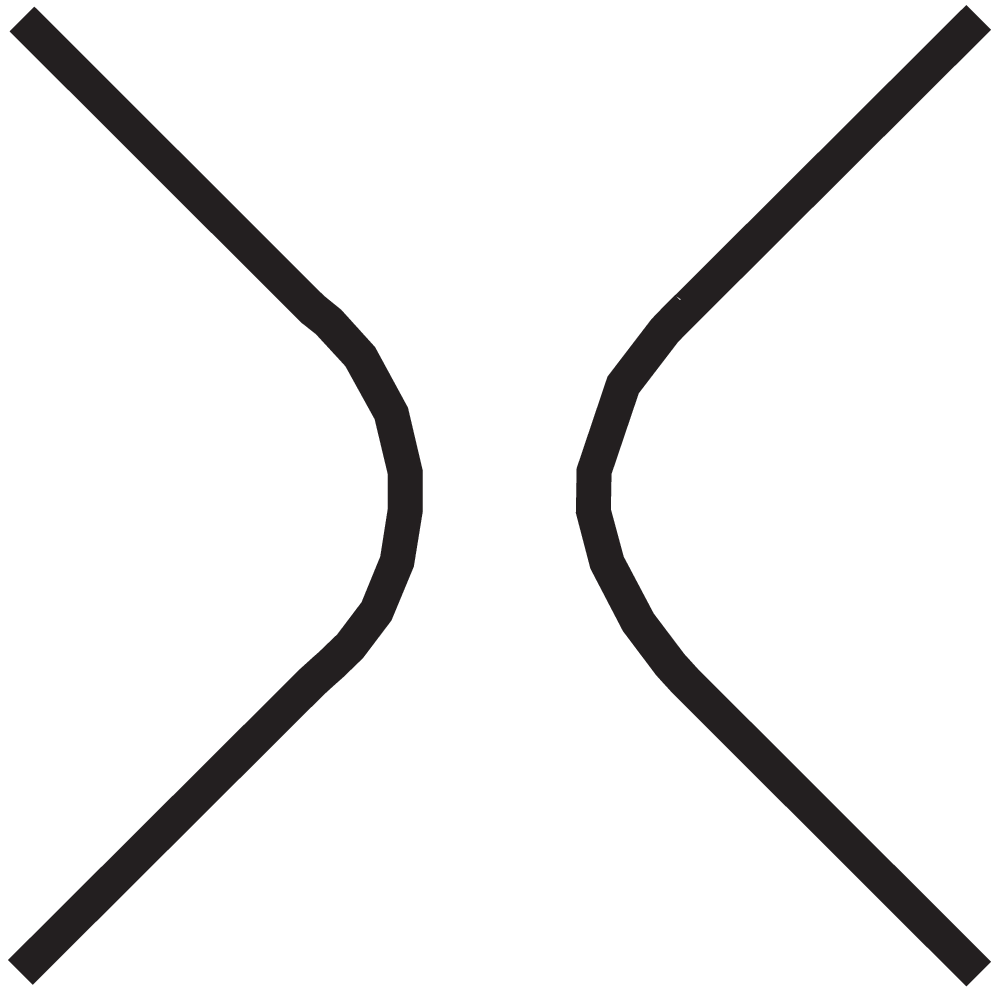}}}

\newcommand{\biqbrnor}{\raisebox{-0.\height}{\centering\includegraphics[width=6cm]{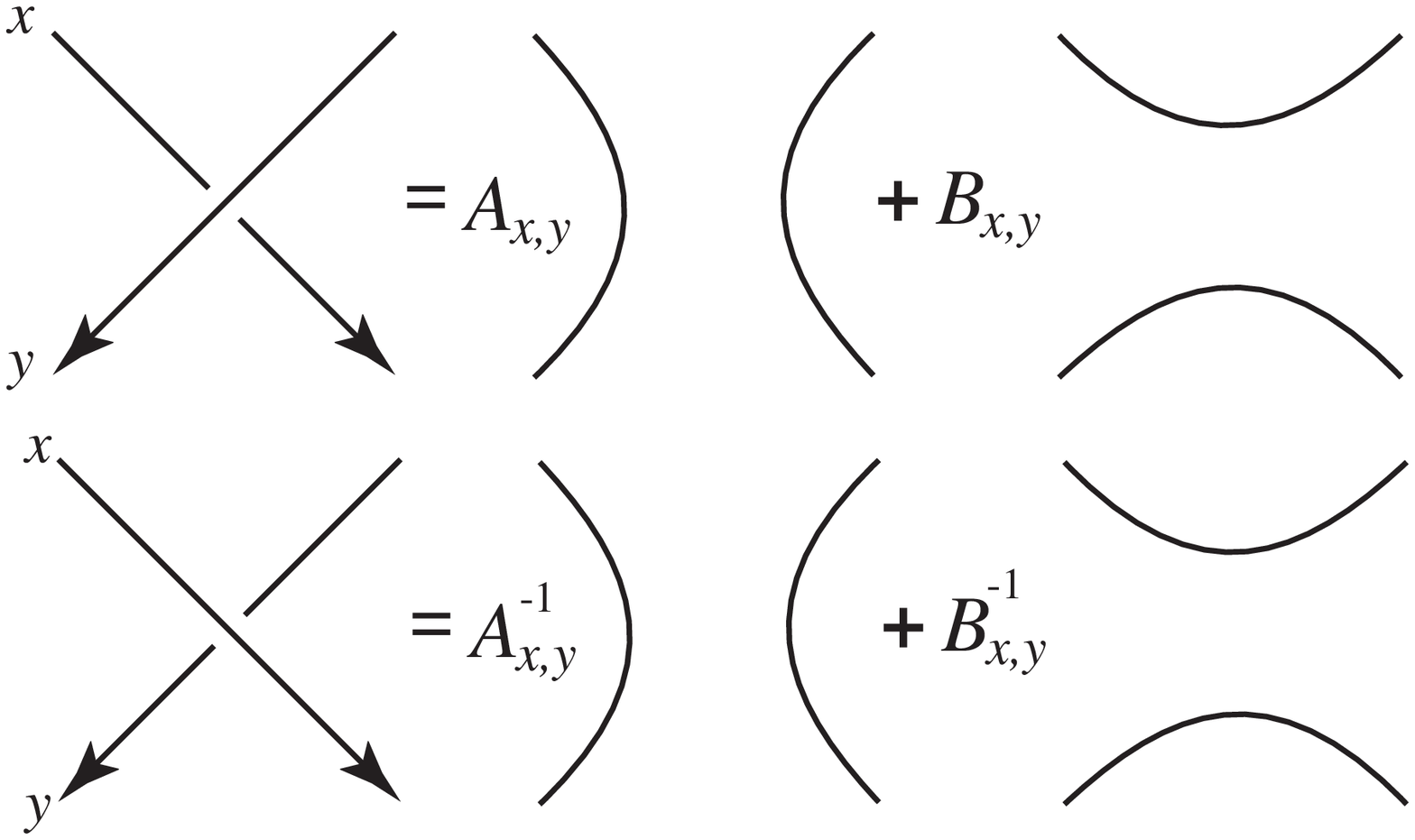}}}
\newcommand{\biqbrdef}{\raisebox{-0.\height}{\centering\includegraphics[width=8cm]{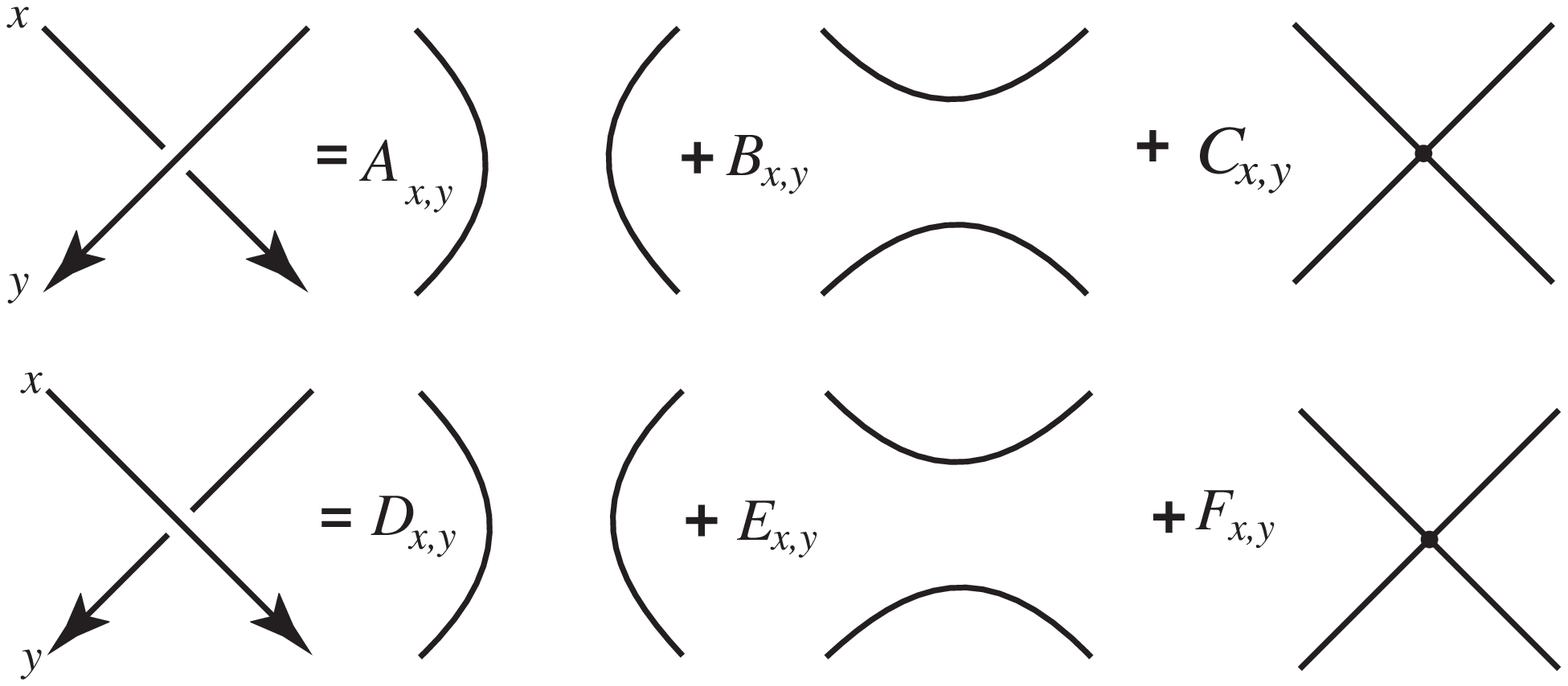}}}
\newcommand{\biqcol}{\raisebox{-0.\height}{\centering\includegraphics[width=6cm]{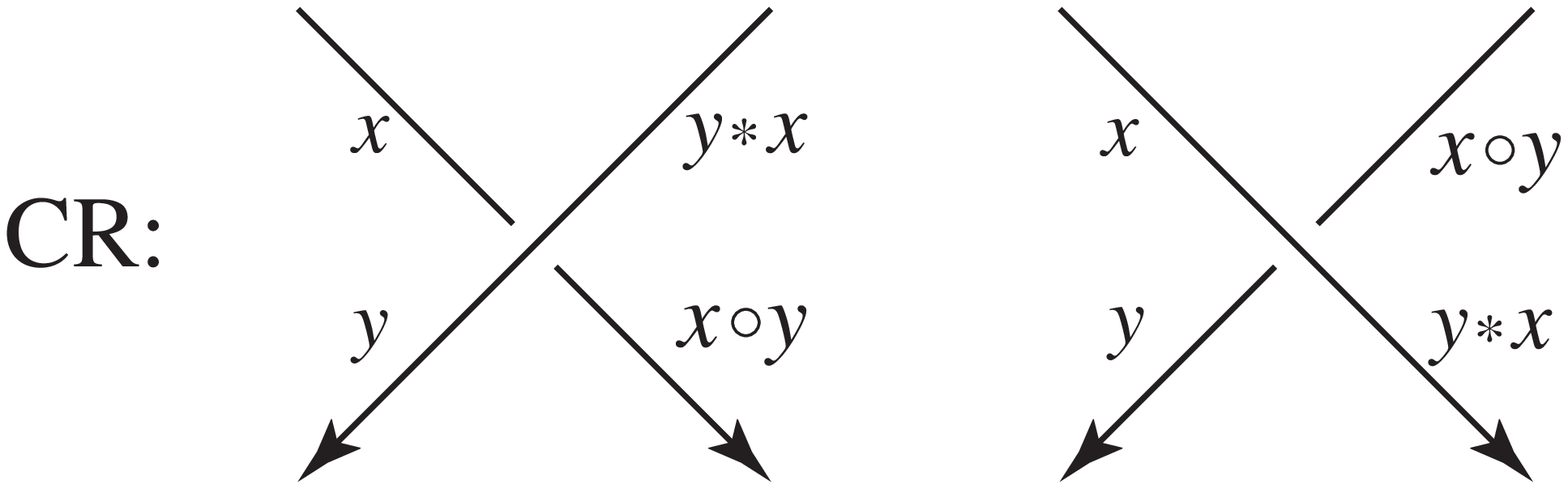}}}
\newcommand{\crev}{\raisebox{-0.\height}{\centering\includegraphics[width=6cm]{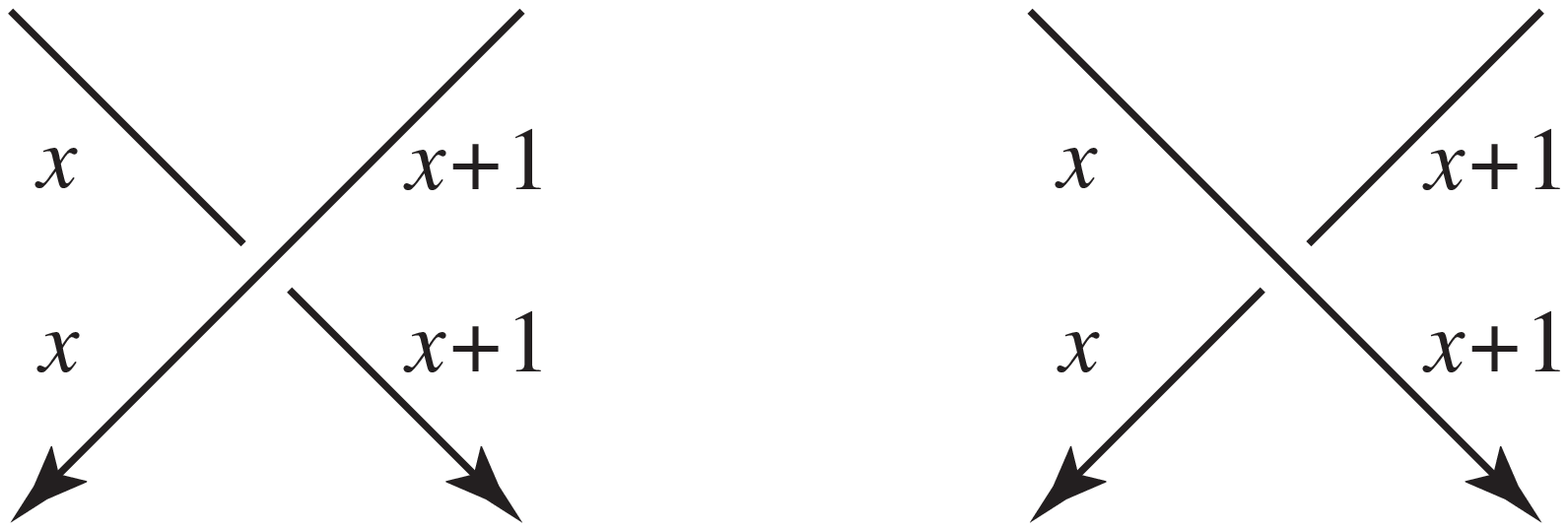}}}
\newcommand{\crod}{\raisebox{-0.\height}{\centering\includegraphics[width=6cm]{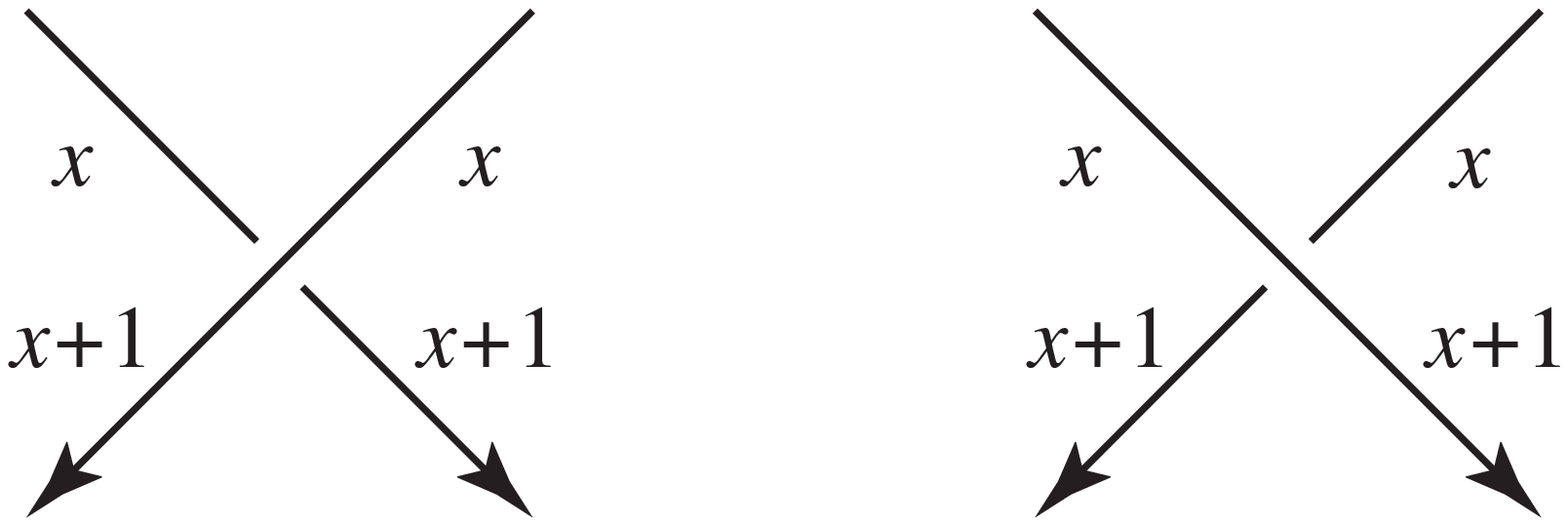}}}

\def\V{\mathcal{V}}
\newcommand{\ob}{\mathrm{ob}}

\title{Picture-valued biquandle bracket}

\author{Denis P.~Ilyutko, Vassily O.~Manturov}

\begin{document}
\date{}

\maketitle

\abstract{In~\cite{Mant31}, the second named author constructed the bracket invariant $[\cdot]$ of virtual knots valued in {\em pictures} (linear combinations of virtual knot diagrams with some crossing information omitted), such that for many diagrams $K$, the following formula holds: $[K]=\widetilde{K}$, where $\widetilde{K}$ is the underlying graph of the diagram, i.e., the value of the invariant on a diagram equals the diagram itself with some crossing information omitted. This phenomenon allows one to reduce many questions about virtual knots to questions about their diagrams.

In~\cite{NOR}, the authors discovered the following phenomenon: having a biquandle colouring of a certain knot, one can enhance various state-sum invariants (say, Kauffman bracket) by using various coefficients depending on colours.

Taking into account that the parity can be treated in terms of biquandles, we bring together the two ideas from these papers and construct the the picture-valued parity biquandle bracket for classical and virtual knots.

This is an invariant of virtual knots valued in pictures. Both the parity bracket and Nelson--Orrison--Rivera invariants are partial cases of this invariants, hence this invariant enjoys many properties of various kinds.}

\notes{0}{
\subclass{57M15, 57M25, 57M27} 
\keywords{Knot, Reidemeister moves, diagram, biquandle, parity, bracket}%
\thank{Research is carried out with the support of Russian Science Foundation (project no. 16-11-10291)}}


 \section{Introduction}


One of the simplest knot invariants is the colouring invariant: one colours edges of a knot diagram by colours from a given palette and counts some colourings which are called admissible. Colouring invariants are a partial case of quandle invariants, which, in turn, are partial cases of biquandle invariants.

The beautiful idea due to Nelson, Orrison, Rivera~\cite{NOR} allows one to use colours in order to enhance various invariants of knots. To this end, one takes a colouring of a knot diagram and a quantum invariant which satisfies a certain skein-relation (say, Kauffman bracket) and tries to take different coefficients for the states. Then one considers a sum over all admissible colourings. In~\cite{NOR}, it is proved that for good quandles this invariant is strictly stronger than the knot group (it detects the Square knot from the Granny knot), the HOMFLY-PT, Jones or Alexander polynomials (it distinguishes $10_{132}$ and $5_1$). Also this invariant detects the right- and left-hand trefoils and hence can distinguish mirror images. Note that this construction can be generalized to virtual links verbatim (we just disregard virtual crossings).

Virtual knots have a very important feature which turns out to be trivial for classical knots (but not links): the {\em parity}, see Sec.~\ref{sec:kn&biq&par}. The simplest (Gaussian) parity is defined in terms of Gauss diagrams: a chord is {\em even} if it is linked with evenly many chords, otherwise it is odd, see Sec.~\ref{sec:kn&biq&par}. Crossings are called {\em even} ({\em odd}) respectively to the chords.

Recall that {\em free knots} are obtained from virtual knots if we omit the over/under information at crossings and replace the cyclic ordering with the cross structure. The parity allows one to realize the following principle for free knots~\cite{IMN2,IMN,IMN3,Mant31,Mant32,Mant33,Mant35,Mant39,Mant36,Mant38,Mant40,Mant41,Mant42,ManIly}:\\ If a free knot diagram is complicated enough then it realizes itself. The latter means that it appears as a subdiagram in any diagram equivalent to it. This principle comes from a very easy formula $[K]=\widetilde{K}$, where $K$ on the LHS is a (virtual or free) knot (i.e., a diagram considered up to various moves), and $\widetilde{K}$ on the RHS is a single diagram (the underlying graph) of the knot (in the case of free knots $\widetilde{K}=K$), which is complicated enough and considered as an element of a linear space formally generated by such diagrams.

The bracket $[\cdot]$ is a {\em diagram-valued} invariant of (virtual, free) knots~\cite{IMN2,IMN,IMN3,Mant31,Mant32,Mant33,Mant35,Mant39,Mant36,Mant38,Mant40,Mant41,Mant42,ManIly}, see Sec.~\ref{subsec:parbr}. For us, it is important to know that
 \begin{enumerate}
  \item[1)]
it is defined by using {\em states} in a way similar to the Kauffman bracket,
  \item[2)]
it is valued not in numbers or (Laurent) polynomials but in {\em diagrams} meaning that we do not completely resolve a knot diagram leaving some crossings intact.
 \end{enumerate}
It is important to note that for some (completely odd) diagrams, no crossings are smoothed at all.

It is the first appearance of diagram-valued invariants in knot theory. For virtual knots, it allows one to make very strong conclusions about the shape of any diagram by looking just at one diagram.

Though there are non-trivial parities for classical knots, parity can be treated in terms of biquandle colourings, see Sec~\ref{sec:kn&biq&par}. Namely, there is a very simple biquandle which allows one to say whether a crossing is even or odd by looking at colours of edges incident to this crossing, see Sec.~\ref{sec:parbiqbr}. Note that a parity is an invariant of {\em crossings} unlike biquandle colouring: if we apply the third Reidemeister move then the corresponding crossings before and after have the same parity, though their colours can differ drastically.

In the papers~\cite{NOR,NR}, (bi)quandles were used to enhance various invariants in the following way: we take some combinatorial scheme for constructing a knot invariant, for example, a quantum invariant, but we use different input data when dealing with crossings: we take into account the colours of the (short) arcs of the knot diagram. Hence, (bi)quandle colourings, though usually not being ``invariants of crossings'', turn out to be a very powerful tool for enhancing knot invariants.

The aim of the present paper is to bring together the ideas from parity theory~\cite{IMN2,IMN,IMN3,Mant31,Mant32,Mant33,Mant35,Mant39,Mant36,Mant38,Mant40,Mant41,Mant42,ManIly} and the ideas from~\cite{NOR} and to construct the universal biquandle picture-valued invariant of classical and virtual links. The invariant is defined here by using state sums but, keeping in mind that there may be even and odd crossings (the latter {\em not} being resolved), we axiomatize the new invariant in such a way that it {\em a priori} dominates both the biquandle bracket and the parity bracket, thus being a very strong invariant of both virtual and classical links: it is at least as strong as the biquandle invariant for classical knots and at least as strong as the parity bracket for virtual knots.

Note that we do not know any partial examples of invariants for classical links, which are {\em diagram-valued}.

The paper is organised as follows. In the next section we give the main definitions concerning knots, biquandles and parity. In Sec.~\ref{sec:parbiqbr} we present a construction of a parity-biquandle bracket.

\section{Knots, biquandle and parity}\label{sec:kn&biq&par}

 \subsection{Knots}

Throughout the paper we consider a knot from combinatorial point of view.

 \begin{definition}
A {\em $4$-valent graph} is a finite graph with each vertex having degree four.

A $4$-valent graph is called a {\em graph with a cross structure} or a {\em framed $4$-graph} if for every vertex the four emanating half-edges are split into two pairs of half-edges (we have the structure of opposite edges).

A $4$-valent graph is called a {\em graph with a cycle order} if for every vertex the four emanating half-edges are cyclically ordered.

In both cases we also admit diagrams which consist of disjoint unions of the above-mentioned diagrams and several circles with no vertices.
 \end{definition}

 \begin{remark}
Every $4$-valent graph is called a graph with a cycle order is a graph with the cross structure.
 \end{remark}

 \begin{definition}
By a ({\em classical}) {\em link diagram} we mean a plane $4$-valent graph with a cycle order and over/undercrossing structure at each vertex. Half-edges constituting a pair are called {\em opposite}. A diagram is called {\em oriented} if each edge is oriented and opposite edges has the same orientation. The relation of half-edges to be opposite allows one to define the notion of unicursal component and to count the number of unicursal components of a link diagram, see Fig.~\ref{knlivir} for classical knots (links with one unicursal component) and links. Vertices of a diagram are called {\em crossings}.

A {\em virtual link diagram} is a generic immersion of a $4$-valent graph into the plane such that we have a cycle order and over/undercrossing structure at each vertex and mark each edge intersection by a circle, see Fig.~\ref{knlivir}. For virtual links the definitions of components and of an oriented virtual link are the same as for classical links.
 \end{definition}

 \begin{figure}
  \centering\includegraphics[width=280pt]{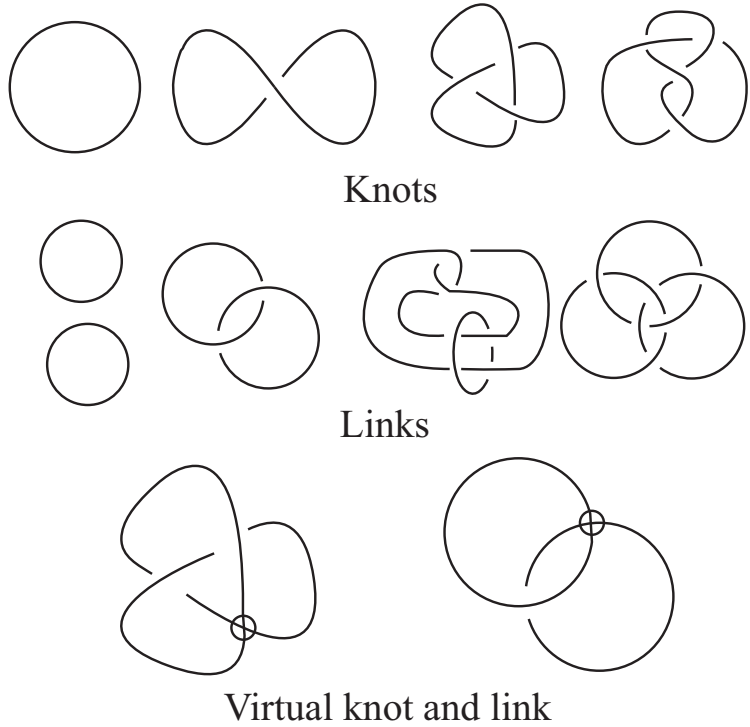}
  \caption{The simplest knots}\label{knlivir}
 \end{figure}

 \begin{definition}
A {\em classical} ({\em virtual}) {\em link} is an equivalence class of classical (virtual) diagrams modulo planar isotopies (diffeomorphisms of the plane on itself preserving the orientation of the plane) and Reidemeister moves (generalized Reidemeister moves). The generalized Reidemeister moves consist of usual Reidemeister moves referring to classical crossings, see Fig.~\ref{rms}, and the {\em detour move} that replaces one arc containing only virtual intersections and self-intersections by another arc of such sort in any other place of the plane, see Fig.~\ref{detour}.

An {\em oriented classical} ({\em virtual}) {\em link} is an equivalence class of oriented classical (virtual) diagrams modulo planar isotopies (diffeomorphisms of the plane on itself preserving the orientation of the plane) and oriented Reidemeister moves (oriented generalized Reidemeister moves), see Fig.~\ref{rdmst} for oriented Reidemeister moves. The oriented generalized Reidemeister moves consist of usual oriented Reidemeister moves and the detour move.
 \end{definition}

 \begin{remark}
It is worth saying that classical knot theory embeds in virtual knot theory~\cite{Mant35,Mant41} and this fact is not trivial.
 \end{remark}

 \begin{figure}
  \centering\includegraphics[width=300pt]{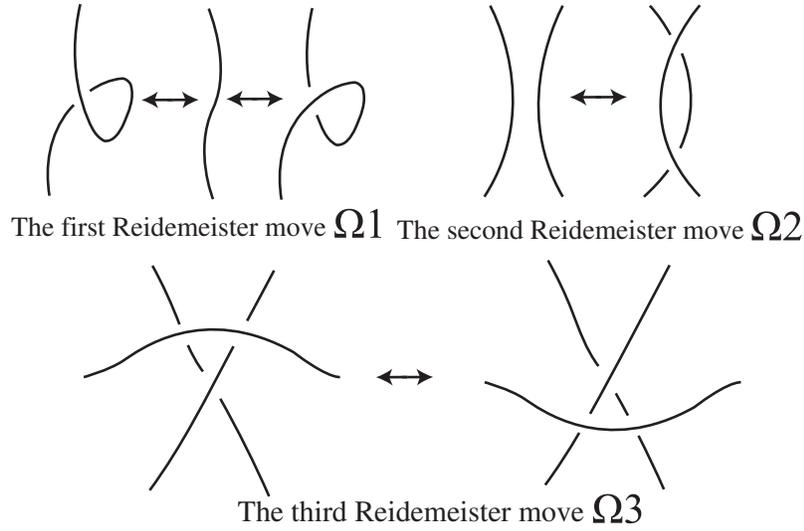}
  \caption{Reidemeister moves $\Omega{1},\,\Omega{2},\,\Omega{3}$}\label{rms}
 \end{figure}

 \begin{figure}
  \centering\includegraphics[width=250pt]{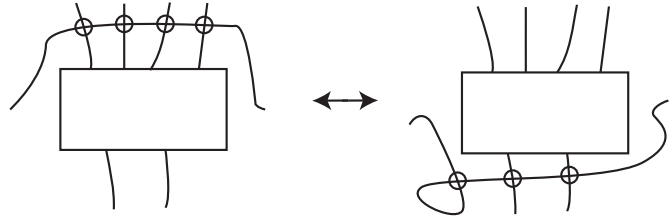}
  \caption{Detour move}\label{detour}
 \end{figure}

 \subsection{Biquandle}

 \begin{definition}\label{def:biq}
A {\em biquandle}~\cite{FJK,HrKa2,KaeKa2,KM1,NOR,NR} is a set $X$ with two binary operations $\circ,\ast\colon X\times X\to X$ satisfying the following axioms:
  \begin{enumerate}
   \item[(R1)]
$x\circ x=x\ast x$ for $\forall\, x\in X$,
   \item[(R2)]
for any $y\in X$ the maps $\alpha_y,\,\beta_y\colon X\to X$ defined by $\alpha_y(x)=x\ast y$, $\beta_y(x)=x\circ y$ are invertible, i.e.\ for any $z_1,\,z_2\in X$ there exist $x_1,\,x_2\in X$ such that $x_1\ast y=z_1$, $x_2\circ y=z_2$,
   \item[(R3)]
the map $S\colon X\times X\to X\times X$ defined by $S(x,y)=(y\ast x,x\circ y)$ is invertible, i.e.\ for any $(z,w)\in X\times X$ there exists $(x,y)\in X\times X$ such that $(y\ast x,x\circ y)=(z,w)$,
   \item[(R4)]
the {\em exchange laws} hols\/:
  \begin{gather*}
(x\circ z)\circ(y\circ z)=(x\circ y)\circ(z\ast y),\\
(y\circ z)\ast(x\circ z)=(y\ast x)\circ(z\ast x),\\
(z\ast x)\ast(y\ast x)=(z\ast y)\ast(x\circ y).
  \end{gather*}
  \end{enumerate}

If $X$ and $Y$ are biquandles then a {\em biquandle homomorphism} is a map $f\colon X\to Y$ such that $f(x*y)=f(x)*f(y)$ and $f(x\circ y)=f(x)\circ f(y)$ for $\forall\,x,\,y\in X$.
 \end{definition}

 \begin{definition}
Let $X$ ba a finite biquandle and $L$ be an oriented link diagram with $n$ crossings.

The {\em fundamental biquandle} $\mathcal{B}(L)$ of $L$ is the set with biquandle operations consisting of equivalence classes of words in a set of generators corresponding to the edges of $L$ modulo the equivalence relation generated by the crossing relations
 \begin{center}
\biqcol
  \end{center}
of $L$ and the biquandle axioms:
 $$
\mathcal{B}(L)=\langle x_1,\dots,x_{2n}\,|\,\mathrm{CR},\,\mathrm{R1},\,\mathrm{R2},\,\mathrm{R3},\,\mathrm{R4}\rangle.
 $$

A {\em biquandle coloring} (or an {\em $X$-coloring}\/) of $L$ is an assignment of elements of $X$ to the edges in $L$ such that the crossing relations are satisfied at every crossing, i.e.\ it is a biquandle homomorphism $f\colon\mathcal{B}(L)\to X$. The set of biquandle colorings of $L$ is identified with the set $\mathrm{Hom}(\mathcal{B}(L),X)$ of biquandle homomorphisms from the fundamental biquandle of $L$ to $X$.
 \end{definition}

 \begin{remark}
The biquandle axioms are the conditions required for every biquandle coloring of the edges in a knot diagram before a move and after the move, see Fig.~\ref{rdmst}. Therefore, the number of biquandle colorings is an invariant.
 \end{remark}

 \begin{figure}
  \centering\includegraphics[width=420pt]{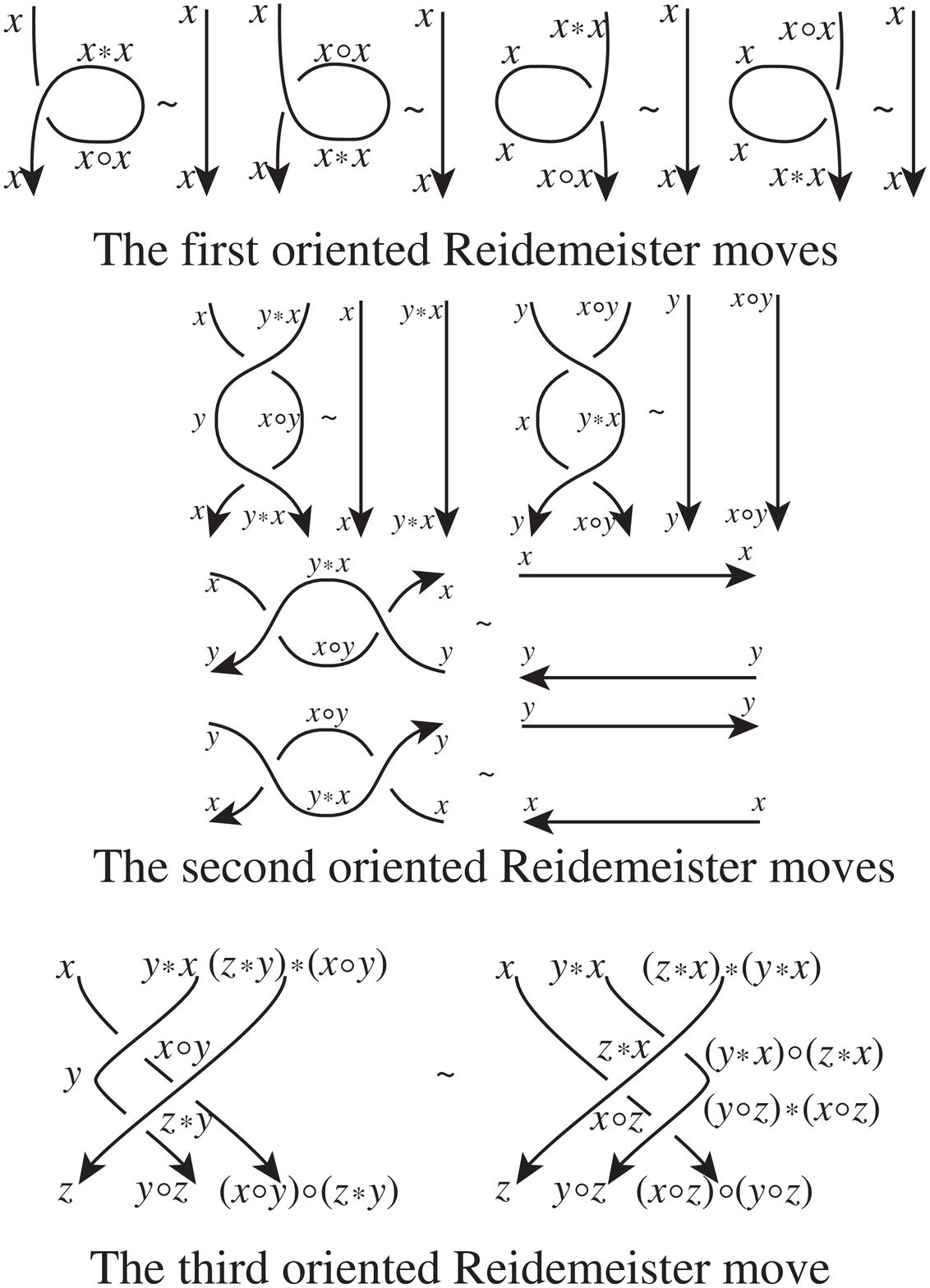}
  \caption{The oriented Reidemeister moves}\label{rdmst}
 \end{figure}

 \subsection{Parity}

Now we show that the Gaussian parity, see~\cite{IMN2,IMN,IMN3,Mant31,Mant32,Mant33,Mant35,Mant39,Mant36,Mant38,Mant40,Mant41,Mant42,ManIly}, can be recovered from biquandle. First, we define a parity and Gauss diagram.

Let $\mathcal{K}$ be a (classical or virtual) knot. Let us define the category $\mathfrak{K}$ of diagrams of the knot $\mathcal{K}$. The objects of $\mathfrak{K}$ are knot diagrams of $\mathcal{K}$ and morphisms of the category $\mathfrak{K}$ are (formal) compositions of {\em elementary morphisms}. By an {\em elementary morphism} we mean
 \begin{itemize}
  \item
an isotopy of diagram (a diffeomorphism of the plane on itself preserving the orientation of the plane);
  \item
a (classical) Reidemeister move.
 \end{itemize}

 \begin{definition}
A {\em partial bijection} of sets $X$ and $Y$ is a triple $(\widetilde X,\widetilde Y,\phi)$, where $\widetilde X\subset X$, $\widetilde Y\subset Y$ and $\phi\colon\widetilde X\to \widetilde Y$ is a bijection.
 \end{definition}

 \begin{remark}
Since the number of vertices of a diagram may change under Reidemeister moves (it is the case of third Reidemeister moves), there is no bijection between the sets of vertices of two diagrams connected by a sequence of Reidemeister moves. To construct any connection between two sets of vertices we have introduced the notion of a partial bijection which means just the bijection between the subsets of vertices corresponding to each
other in the two diagrams.
 \end{remark}

Let us denote by $\V$ the {\em vertex functor} on $\mathfrak{K}$, i.e.\ a functor from $\mathfrak{K}$ to the category, which objects being finite sets and morphisms are partial bijections. For each knot diagram $K$ we define $\V(K)$ to be the set of classical crossings of $K$, i.e.\ the vertices of the underlying $4$-valent graph. Any elementary morphism $m\colon K\to K'$ naturally induces a partial bijection $m_*\colon\V(K)\to\V(K')$.

Let $A$ be an abelian group.

 \begin{definition}\label {def:parity}
A {\em parity $p$ on diagrams of a knot $\mathcal{K}$ with coefficients in $A$} is a family of maps $p_K\colon \V(K)\to A$, $K\in\ob(\mathfrak{K})$, such that for any elementary morphism $m\colon K\to K'$ the following holds:
 \begin{itemize}
  \item[1)]
$p_{K'}(m_*(v))=p_K(v)$ provided that $v\in\V(K)$ and there exists $m_*(v)\in\V(K')$;
  \item[2)]
$p_K(v)=0$ if $m$ is a decreasing first Reidemeister move applied to $K$ and $v$ is the disappearing crossing of $K$;
   \item[3)]
$p_K(v_1)+p_K(v_2)=0$ if $m$ is a decreasing second Reidemeister move and $v_1,\,v_2$ are the disappearing crossings;
   \item[4)]
$p_K(v_1)+p_K(v_2)+p_K(v_3)=0$ if $m$ is a third Reidemeister move and $v_1,\,v_2,\,v_3$ are the crossings participating in this move.
  \end{itemize}
 \end{definition}

 \begin{remark}
Note that each knot may have its own group $A$, and, therefore, different knots generally have different parities.

We will be interested in the group $\mathbb{Z}_2$, because interesting examples of parity have this group.
 \end{remark}

Since Gau{\ss} time it is known that knots can be encoded by chord diagrams with certain additional information. If we view a knot diagram as an immersion of the circle into the plane, we can obtain a {\em chord diagram} by joining those points on the circle which are mapped to the same point under the immersion, by chords.

 \begin{definition}
A {\em Hamiltonian cycle} on a graph is a cycle passing through all vertices of the graph. By a {\em chord diagram} we mean a cubic graph consisting of one selected non-oriented Hamiltonian cycle ({\em core circle} or {\em circle}\/) and a set of non-oriented edges ({\em chords}\/) connecting points on the cycle, moreover, distinct chords have no common points on the cycle.

We say that two chords of a chord diagram are {\em linked} if the ends of one chord belong to two different connected components of the complement to the ends of the other chord in the core circle. Otherwise, we say that chords are {\em unlinked}.
 \end{definition}

 \begin{remark}
As a rule, a chord diagram is depicted on the plane as the Euclidean circle with a collection of chords connecting end points of chords (intersection points of chords which appear as artifacts of drawing chords do not count as vertices).
 \end{remark}

Having a knot diagram (classical or virtual) one can assign to it a chord diagram with an additional structure, the {\em Gauss
diagram}~\cite{GPV,PV}.

 \begin{definition}
Let $K$ be an oriented knot diagram. Let us fix a point on the diagram distinct from a vertex of the diagram. The {\em Gauss diagram} $G(K)$ corresponding to $K$ is the chord diagram consisting of the core circle (with a point, distinct from a vertex, fixed) on which the preimages of the overcrossing and the undercrossing for each crossing are connected by an arrow oriented from the preimage of the overcrossing to the preimage of the undercrossing. Moreover, each arrow is endowed with a sign equal to the sign of the crossing, i.e., the sign is equal to $1$ for a crossing $\skcrro$ and $-1$ for a crossing $\skcrlo$.
 \end{definition}

The Gauss diagram of the right-handed trefoil is shown in Fig.~\ref{gdiagr}.

 \begin{figure}
\centering\includegraphics[width=250pt]{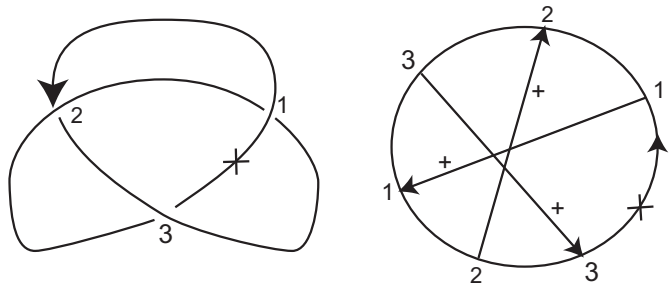} \caption{The Gauss diagram of the right-handed trefoil} \label{gdiagr}
 \end{figure}

 \begin{remark}
It is not difficult to rewrite the Reidemeister moves in the language of Gauss diagrams. As a result, we can think of a knot as
an equivalence class of Gauss diagrams modulo formal moves.

If we consider an arbitrary chord diagram and try to depict a classical knot corresponding to it (for any orientation and signs of
chords) then, in many cases, we shall fail, see~\cite{CE1,CE2,Mant22,RR}. Let us consider the chord diagram $D$ depicted in Fig.~\ref {chdiag}. It is easy to see that there is no knot diagram such that Gauss diagram of which is $D$ for whatever
choice of orientations and signs of chords~\cite{ViK}.

Admitting arbitrary chord diagrams (with arrows and signs) as Gauss diagrams and the Reidemeister moves as moves on all chord diagrams, we exactly get a new theory, the theory of virtual knots~\cite{GPV}.
 \end{remark}

 \begin{figure}
\centering\includegraphics[width=100pt]{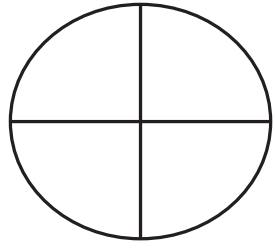}
\caption{``Non-realisable'' chord diagram} \label{chdiag}
 \end{figure}

Let us consider some examples of parities for some knot theories. Parities were first constructed for free knots (a simplification of virtual knots, at each crossing we remember only the structure of opposite edges).

 \begin{example}
Let $A=\mathbb{Z}_2$ and $K$ be a virtual knot diagram. Let us construct the map $gp_K\colon\V(K)\to\mathbb{Z}_2$ by putting
$gp_K(v)=0$ if the number of chords linked with the chord corresponding to $v$ in the Gauss diagram $G(K)$ is even (an {\em even crossing}), and $gp_K(v)=1$ otherwise (an {\em odd crossing}).

 \begin{statement}[see~\cite{Mant35}]
The map $gp$ is a parity for virtual knots.
 \end{statement}

 \begin{definition}
The parity $gp$ is called the {\em Gaussian parity}.
 \end{definition}
 \end{example}

 \begin{example}
Here we call the theory of two-component virtual links just knot theory.

Let $A=\mathbb{Z}_2$ and $L=L_{1}\cup L_{2}$ be a virtual (or classical) link diagram. Let us define the map $p_L\colon\V(L)\to\mathbb{Z}_2$ by putting $p_L(v)=0$ if  $v$ is a classical crossing formed by a single component, and
$p_L(v)=1$ if  $v$ is a classical crossing formed by two components.

 \begin{statement}
The map $p$ is a parity for the set of two-component  virtual links.
 \end{statement}
 \end{example}

 \begin{example}\label{ex:parbiq}
Let $K$ be an oriented knot diagram. We take the set $X=\mathbb{Z}_{2}$ consisting of two elements with the operations $x\circ y=x*y=x+1\pmod 2$.

Let us construct the map $bp_K\colon\V(K)\to\mathbb{Z}_2$ by putting $bp_K(v)=0$ if the crossing has the structure
 \begin{center}
\crev
 \end{center}
(an {\em even crossing}), and $gp_K(v)=1$ otherwise, i.e.\ it has the structure
 \begin{center}
 \crod.
 \end{center}
(an {\em odd crossing}).

The following theorem follows from a direct check.

 \begin{theorem}
The map $bp$ is a parity and coincides with the map $gp$.
 \end{theorem}

 \end{example}

\section{Biquandle, parity and parity-biquandle bracket}\label{sec:parbiqbr}

 \subsection{The parity bracket}\label{subsec:parbr}

The first example of the parity bracket appeared in~\cite{Mant31}. That bracket was constructed for the Gaussian parity and played a significant role in, for instance, proving minimality theorems, reducing problems about diagrams to questions about graphs etc. Also the bracket was generalised for the case of graph-links, see~\cite{IM2}, and allowed the authors to prove the existence of non-realisable graph-links, for more details see~\cite{IM2}.

In this subsection we consider the parity bracket for any parity valued in $\mathbb{Z}_2$ (see~\cite{IMN2,IMN,IMN3,Mant31,Mant32,Mant33,Mant35,Mant39,Mant36,Mant38,Mant40,Mant41,Mant42,ManIly}). This bracket is a generalisation of the bracket from~\cite{Mant31}.

  \begin{definition}
Let $\alpha$ be a formal variable. Let $\mathfrak{G}_\alpha$ be the set of all equivalence classes of $4$-valent graphs with a cross structure modulo the following equivalence relations:
 \begin{enumerate}
  \item[1)]
the second Reidemeister move, see Fig.~\ref{rdmst2fn};
  \item[2)]
$L\sqcup \bigcirc=\alpha L$.
 \end{enumerate}
Consider the linear space $\mathbb{Z}_2\mathfrak{G}_\alpha$.
 \end{definition}

 \begin{figure}
  \centering\includegraphics[width=170pt]{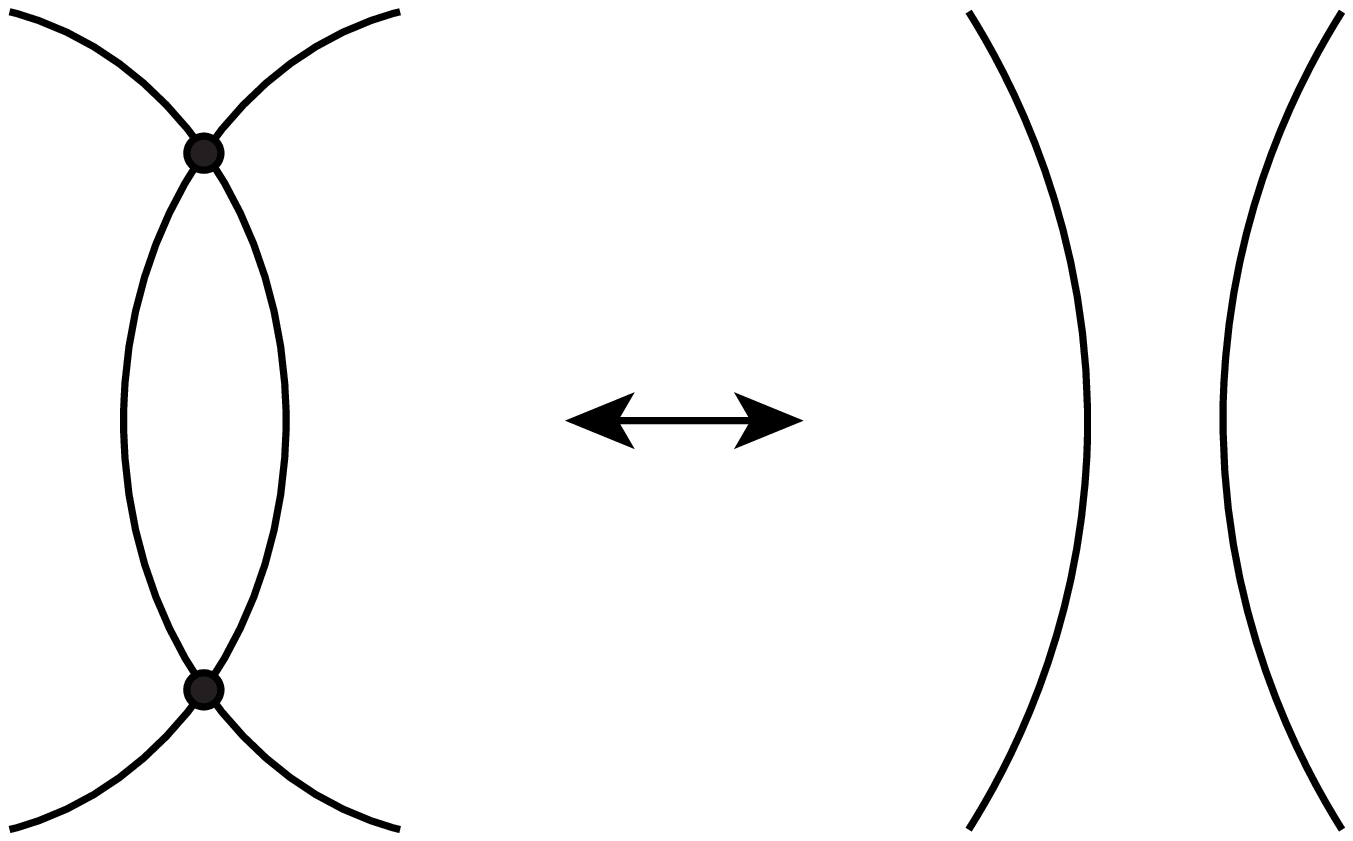}
  \caption{Second Reidemeister move}\label{rdmst2fn}
 \end{figure}

 \begin{remark}
It is worth saying that the recognition problem is easy solvable when we have only second Reidemeister moves. Indeed, each equivalence class of $4$-valent graphs with cross structure has a unique minimal representative which can be obtained from any representative of this graph by a ``monotonous descent algorithm.''
 \end{remark}

Let $\mathcal K$ be a virtual knot, $p$ be a parity on diagrams of $\mathcal K$ with coefficients from the group $\mathbb{Z}_2$, and $K$ be a diagram of $\mathcal K$ with $\V(K)=\{v_1,\dots,v_n\}$. For each element $s=(s_1,\dots,s_n)\in\{0,1\}^n$ we define $K_s$ to be equal to the sum of $2^l$ graphs obtained from $K$ by a smoothing, the positive smoothing $\skcrossr\to\skcrh$ or the negative smoothing $skcrossr\to\skcrv$) at each vertex $v_i$ if $s_i=1$, where $l$ is the number of $1$ in the $(s_1,\dots,s_n)$. Define $q_{K,s}(v_i)=p_K(v_i)$ if $s_i=0$, and $q_{K,s}(v_i)=1-p_K(v_i)$ if $s_i=1$.

 \begin{definition}
The {\em parity bracket} $[K]$ is the following sum:
 $$
\sum\limits_{s\in\{0,1\}^n}\prod\limits_{i=1}^nq_{K,s}(v_i)
K_s\in\mathbb{Z}_2\mathfrak{G}_0.
 $$
 \end{definition}

 \begin{remark}
In fact, in the definition of the parity bracket we take only those summands where even crossings are smoothed.
 \end{remark}

 \begin{theorem}[see~\cite{IMN2,IMN,IMN3,Mant31,Mant32,Mant33,Mant35,Mant39,Mant36,Mant38,Mant40,Mant41,Mant42,ManIly}]
If two virtual knot diagrams $K$ and $K'$ represent the same knot then the following equality holds in $\mathbb{Z}_2\mathfrak{G}${\em:}
 $$
[K]=[K'].
 $$
 \end{theorem}

 \begin{remark}
In particular, if $K$ is irreducible and odd (see definitions in~\cite{IMN2,IMN,IMN3,Mant31,Mant32,Mant33,Mant35,Mant39,Mant36,Mant38,Mant40,Mant41,Mant42,ManIly}) then $[K]=\widetilde{K}$, hence $K$ can be obtained from any equivalent diagram $K'$ by smoothing $K'$ at some crossings.
 \end{remark}

 \subsection{The biquandle bracket}

Let us recall a key example from~\cite{NOR}. Let $L$ be an oriented virtual link diagram and let us colour its edges by a finite biquandle $(X,\circ,\ast)$ (in~\cite{NOR} $\circ$ and $\ast$ are denoted by $\unrhd$ and $\overline{\rhd}$ respectively). Further, we fix a commutative ring $R$ with unity and denote the set of units of $R$ by $R^\prime$. Then for such a colouring we can get the following set of skein relations:
 \begin{center}
\biqbrnor
 \end{center}
(we disregard virtual crossings and consider the resulting states as collections of intersecting curves on the plane).

Here $A,\,B\colon X\times X\to R^\prime$ and $A_{x,y}=A(x,y),\,B_{x,y}=B(x,y)$. In order that the skein relations give us a link invariant we have to consider two distinguished elements $\delta\in R$ being the value of a simple closed curve (virtual crossings are disregarded) and $w\in R^\prime$ being the value of a positive kink ($w^{-1}$ being the value of a negative kink) and $A,\,B,\delta,w$  should satisfy the following relations:
 \begin{enumerate}
  \item[(i)]
$\delta A_{x,x}+B_{x,x}=w$ and $\delta A^{-1}_{x,x}+B^{-1}_{x,x}=w^{-1}$ for $\forall\,x\in X$,
  \item[(ii)]
$\delta=-A_{x,y}B^{-1}_{x,y}-A^{-1}_{x,y}B_{x,y}$ for $\forall\,x,\,y\in X$,
  \item[(iii)]
   \begin{align*}
A_{x,y}A_{y,z}A_{x\circ y,z\ast y}=&A_{x,z}A_{y\ast x,z\ast x}A_{x\circ z,y\circ z},\\
A_{x,y}B_{y,z}B_{x\circ y,z\ast y}=&B_{x,z}B_{y\ast x,z\ast x}A_{x\circ z,y\circ z},\\
B_{x,y}A_{y,z}B_{x\circ y,z\ast y}=&B_{x,z}A_{y\ast x,z\ast x}B_{x\circ z,y\circ z},\\
A_{x,y}A_{y,z}B_{x\circ y,z\ast y}=&A_{x,z}B_{y\ast x,z\ast x}A_{x\circ z,y\circ z}+
A_{x,z}A_{y\ast x,z\ast x}B_{x\circ z,y\circ z}\\
&+\delta A_{x,z}B_{y\ast x,z\ast x}B_{x\circ z,y\circ z}+B_{x,z}B_{y\ast x,z\ast x}B_{x\circ z,y\circ z},\\
B_{x,z}A_{y\ast x,z\ast x}A_{x\circ z,y\circ z}=&B_{x,y}A_{y,z}A_{x\circ y,z\ast y}+A_{x,y}B_{y,z}A_{x\circ y,z\ast y}\\
&+\delta B_{x,y}B_{y,z}A_{x\circ y,z\ast y}+B_{x,y}B_{y,z}B_{x\circ y,z\ast y}
   \end{align*}
for $\forall\,x,\,y,\,z\in X$.
 \end{enumerate}

 \begin{definition}
A pair of maps $A,\,B\colon X\times X\to R$ together with elements $\delta\in R$ and $w\in R^\prime$ satisfying the above relations is called a {\em biquandle bracket} (or an {\em $X$-bracket}).
 \end{definition}

Let $L$ be an oriented (virtual) link diagram with $n$ crossings and let
 $$
\mathcal{B}(L)=\langle x_1,\dots,x_{2n}\,|\,\mathrm{CR},\,\mathrm{R1},\,\mathrm{R2},\,\mathrm{R3},\,\mathrm{R4}\rangle.
 $$
be its fundamental biquandle. There are $2^n$ states of $L$, i.e.\ at each crossing we have either the positive smoothing or the negative smoothing (we disregard virtual crossings). For each state we have the product of $n$ factors of $A^{\mp1}_{x,y}$, or $B^{\mp1}_{x,y}$ times $\delta^k$, where $k$ is the number of circles in the state.

 \begin{definition}
The {\em fundamental biquandle bracket value} for $L$ is the sum of the contributions times $w^{-\mathrm{wr}(L)}$, where $\mathrm{wr}(L)$ is the writhe number of $L$.
 \end{definition}


 \begin{definition}
Let $X$ be a finite biquandle, $f\in\mathrm{Hom}(\mathcal{B}(L),X)$ be an $X$-coloring and $\beta$ be an $X$-bracket. We set the value of the fundamental biquandle bracket value in $f$ for $\beta$ to be $\beta(f)\in R$.

The {\em biquandle bracket multiset} of $L$ is the following multiset:
 $$
\Phi^{\beta,M}_X(L)=\{\beta(f)\,|f\in\mathrm{Hom}(\mathcal{B}(L),X)\}.
 $$

The {\em biquandle bracket polynomial} of $L$ is
 $$
\Phi^{\beta}_X(L)=\sum\limits_{f\in\mathrm{Hom}(\mathcal{B}(L),X)}u^{\beta(f)}.
 $$
 \end{definition}

 \begin{theorem}[see~\cite{NOR}]
The biquandle bracket multiset and the biquandle bracket polynomial are invariants of virtual links.
 \end{theorem}

 \subsection{The parity-biquandle bracket}

Now we are going to bring together the two main ideas: picture-valued invariants and biquandle brackets. For this we will make a distinction between different types of crossings and, having this in mind, instead of the two states we consider three states by adding the state with a graphical vertex. Moreover, the value of the biquandle bracket will be an equivalence class of linear combination of $4$-valent graphs with some coefficients in a commutative ring $R$ with unity.

Let $L$ be an oriented virtual link diagram. Let us consider the following set of skein relations:
 \begin{center}
\biqbrdef
  \end{center}
(virtual crossings are disregarded). Here $A,\,B,\,C,\,D,\,E,\,F\colon X\times X\to R$ and $A_{x,y}=A(x,y)$, $B_{x,y}=B(x,y)$, $C_{x,y}=C(x,y)$, $D_{x,y}=D(x,y)$, $E_{x,y}=E(x,y)$, $F_{x,y}=F(x,y)$.

Our purpose is to write down a whole set of relations such that the skein relations give rise to a picture-valued link invariant. Let us again consider two distinguished elements $\delta\in R$ being the value of a simple closed curve (virtual crossings are disregarded) and $w\in R^\prime$ being the value of a positive kink ($w^{-1}$ being the value of a negative kink) and, moreover, diagrams having graphical vertices are considered modulo the second Reidemeister moves, see Fig.~\ref{rdmst2fn}. The value of the new bracket lies in the module $R\mathfrak{G}_\delta$.

Let us apply skein relations to the oriented Reidemeister moves, see Fig.~\ref{rdmst}. As a result we will get some relations on the coefficients.

The first Reidemeister moves give us the following relations:
 $$
\delta A_{x,x}+B_{x,x}=w,\quad \delta D_{x,x}+E_{x,x}=w^{-1},\quad C_{x,x}=F_{x,x}=0.
 $$
for $\forall\,x\in X$, see Fig.~\ref{bbr1}.

 \begin{figure}
  \centering\includegraphics[width=400pt]{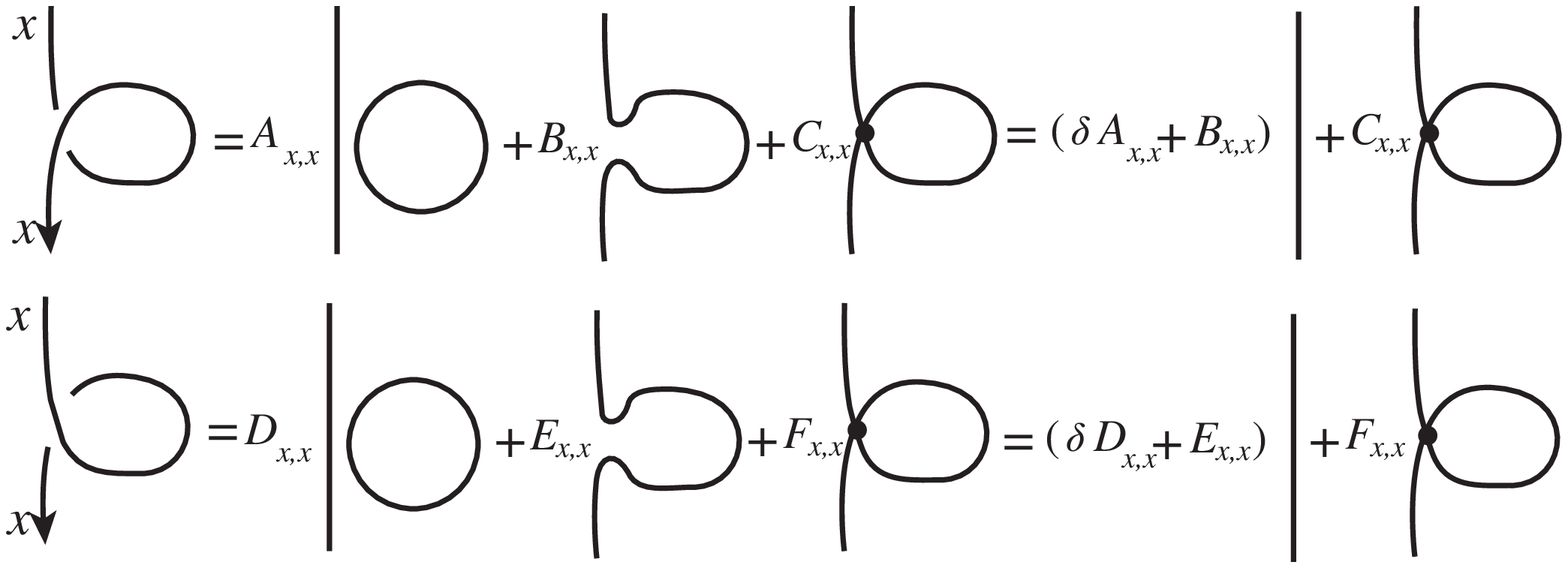}
  \caption{Relations from $\Omega{1}$}\label{bbr1}
 \end{figure}

The first two versions of second Reidemeister moves, in which the strands are oriented in the same direction, give us the following relations:
 \begin{gather*}
A_{x,y}D_{x,y}+C_{x,y}F_{x,y}=1,\quad B_{x,y}F_{x,y}=C_{x,y}E_{x,y}=0,\\
A_{x,y}F_{x,y}+C_{x,y}D_{x,y}=
A_{x,y}E_{x,y}+B_{x,y}D_{x,y}+\delta B_{x,y}E_{x,y}=0
 \end{gather*}
for $\forall\,x,\,y\in X$, see Fig.~\ref{bbr21}.

 \begin{figure}
  \centering\includegraphics[width=400pt]{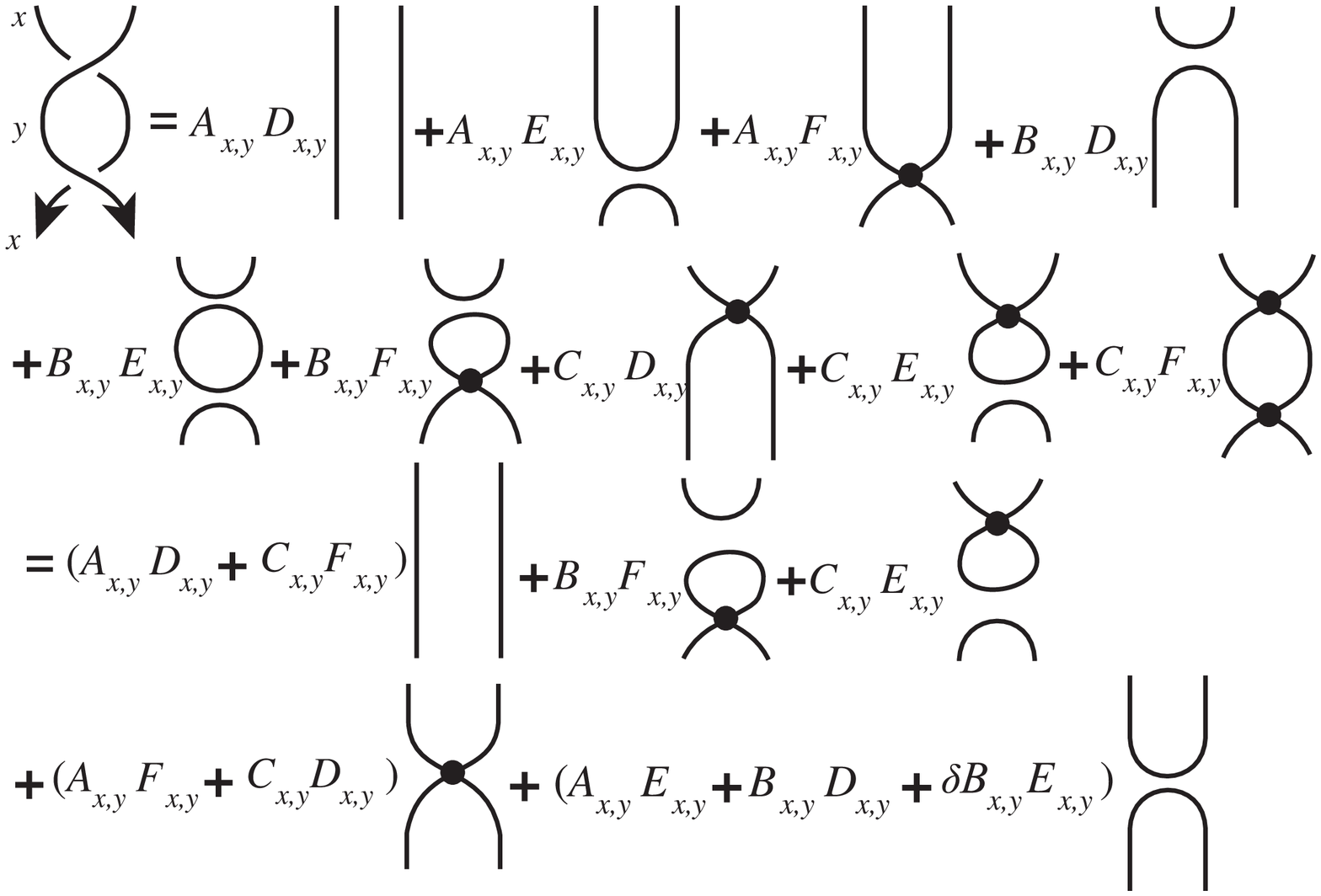}
  \caption{Relations from $\Omega{2}$}\label{bbr21}
 \end{figure}

The last two versions of second Reidemeister moves, in which the strands are oriented in opposite directions, give us the following additional relations:
 \begin{gather*}
B_{x,y}E_{x,y}+C_{x,y}F_{x,y}=1,\quad A_{x,y}F_{x,y}=C_{x,y}D_{x,y}=0,\\
B_{x,y}F_{x,y}+C_{x,y}E_{x,y}=
A_{x,y}E_{x,y}+B_{x,y}D_{x,y}+\delta A_{x,y}D_{x,y}=0
 \end{gather*}
for $\forall\,x,\,y\in X$, see Fig.~\ref{bbr22}.

 \begin{figure}
  \centering\includegraphics[width=400pt]{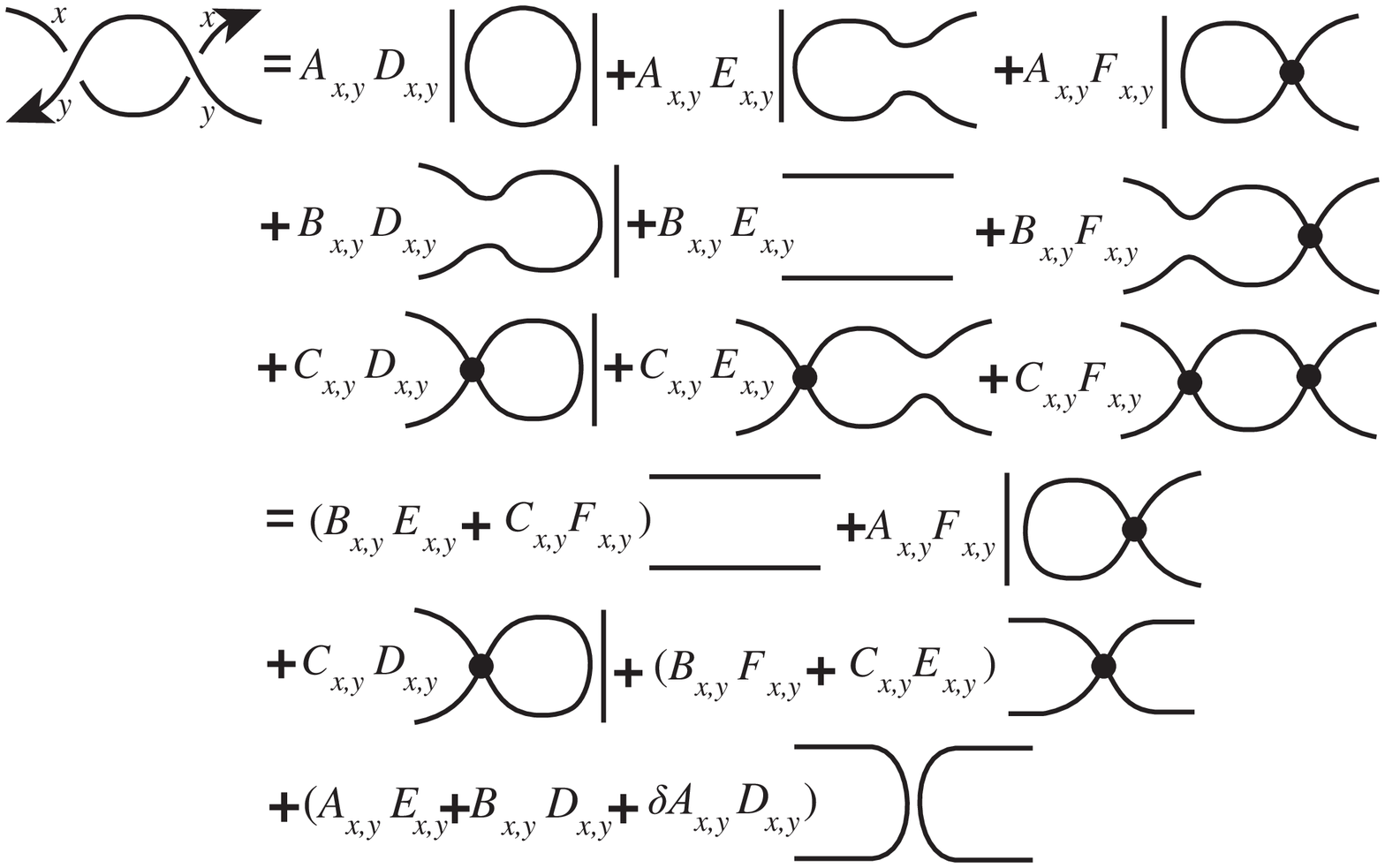}
  \caption{Relations from $\Omega{2}$}\label{bbr22}
 \end{figure}

The third Reidemeister moves give us the following relations:
 \begin{gather*}
A_{x,y}A_{y,z}A_{x\circ y,z\ast y}+C_{x,y}C_{y,z}A_{x\circ y,z\ast y}=A_{x,z}A_{y\ast x,z\ast x}A_{x\circ z,y\circ z}+A_{x,z}C_{y\ast x,z\ast x}C_{x\circ z,y\circ z},\\
A_{x,y}B_{y,z}B_{x\circ y,z\ast y}+C_{x,y}B_{y,z}C_{x\circ y,z\ast y}=B_{x,z}B_{y\ast x,z\ast x}A_{x\circ z,y\circ z}+C_{x,z}B_{y\ast x,z\ast x}C_{x\circ z,y\circ z},\\
B_{x,y}A_{y,z}B_{x\circ y,z\ast y}+B_{x,y}C_{y,z}C_{x\circ y,z\ast y}=B_{x,z}A_{y\ast x,z\ast x}B_{x\circ z,y\circ z}+C_{x,z}C_{y\ast x,z\ast x}B_{x\circ z,y\circ z},\\
A_{x,y}C_{y,z}A_{x\circ y,z\ast y}+C_{x,y}A_{y,z}A_{x\circ y,z\ast y}=C_{x,z}A_{y\ast x,z\ast x}A_{x\circ z,y\circ z},\\
A_{x,y}A_{y,z}C_{x\circ y,z\ast y}=A_{x,z}A_{y\ast x,z\ast x}C_{x\circ z,y\circ z}+A_{x,z}C_{y\ast x,z\ast x}A_{x\circ z,y\circ z},\\
A_{x,y}C_{y,z}B_{x\circ y,z\ast y}=B_{x,z}B_{y\ast x,z\ast x}C_{x\circ z,y\circ z}+C_{x,z}B_{y\ast x,z\ast x}A_{x\circ z,y\circ z},\\
B_{x,y}C_{y,z}B_{x\circ y,z\ast y}+B_{x,z}A_{y,z}C_{x\circ y,z\ast y}=B_{x,z}A_{y\ast x,z\ast x}C_{x\circ z,y\circ z},\\
A_{x,y}B_{y,z}C_{x\circ y,z\ast y}+C_{x,z}B_{y,z}B_{x\circ y,z\ast y}=B_{x,z}C_{y\ast x,z\ast x}A_{x\circ z,y\circ z},\\
C_{x,y}A_{y,z}B_{x\circ y,z\ast y}=B_{x,z}C_{y\ast x,z\ast x}B_{x\circ z,y\circ z}+C_{x,z}A_{y\ast x,z\ast x}B_{x\circ z,y\circ z},\\
C_{x,y}C_{y,z}B_{x\circ y,z\ast y}=B_{x,z}C_{y\ast x,z\ast x}C_{x\circ z,y\circ z},\\
A_{x,y}C_{y,z}C_{x\circ y,z\ast y}=C_{x,z}C_{y\ast x,z\ast x}A_{x\circ z,y\circ z},\\
C_{x,y}A_{y,z}C_{x\circ y,z\ast y}=C_{x,z}A_{y\ast x,z\ast x}C_{x\circ z,y\circ z},\\
B_{x,y}C_{y,z}A_{x\circ y,z\ast y}=C_{x,y}B_{y,z}A_{x\circ y,z\ast y}=B_{x,y}B_{y,z}C_{x\circ y,z\ast y}=C_{x,y}C_{y,z}C_{x\circ y,z\ast y}=0,\\
A_{x,z}B_{y\ast x,z\ast x}C_{x\circ z,y\circ z}=A_{x,z}C_{y\ast x,z\ast x}B_{x\circ z,y\circ z}\\
=C_{x,z}B_{y\ast x,z\ast x}B_{x\circ z,y\circ z}=C_{x,z}C_{y\ast x,z\ast x}C_{x\circ z,y\circ z}=0,\\
A_{x,y}A_{y,z}B_{x\circ y,z\ast y}=A_{x,z}B_{y\ast x,z\ast x}A_{x\circ z,y\circ z}+
A_{x,z}A_{y\ast x,z\ast x}B_{x\circ z,y\circ z}\\
+\delta A_{x,z}B_{y\ast x,z\ast x}B_{x\circ z,y\circ z}+B_{x,z}B_{y\ast x,z\ast x}B_{x\circ z,y\circ z},\\
B_{x,z}A_{y\ast x,z\ast x}A_{x\circ z,y\circ z}=B_{x,y}A_{y,z}A_{x\circ y,z\ast y}+A_{x,y}B_{y,z}A_{x\circ y,z\ast y}\\
+\delta B_{x,y}B_{y,z}A_{x\circ y,z\ast y}+B_{x,y}B_{y,z}B_{x\circ y,z\ast y}
   \end{gather*}
for $\forall\,x,\,y,\,z\in X$, see Fig.~\ref{bbr31},~\ref{bbr32}.

 \begin{figure}
  \centering\includegraphics[width=420pt]{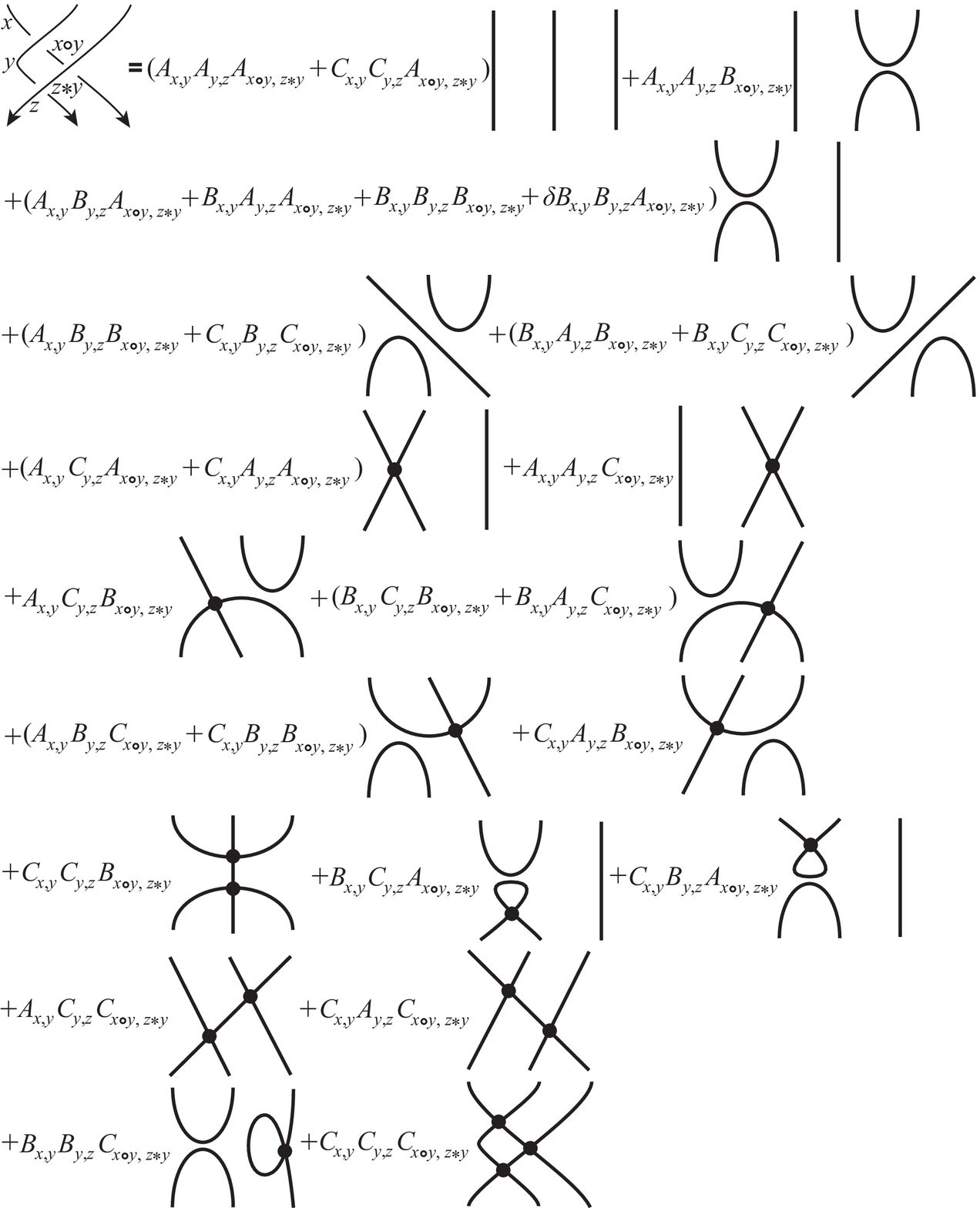}
  \caption{Relations from $\Omega{3}$}\label{bbr31}
 \end{figure}

 \begin{figure}
  \centering\includegraphics[width=440pt]{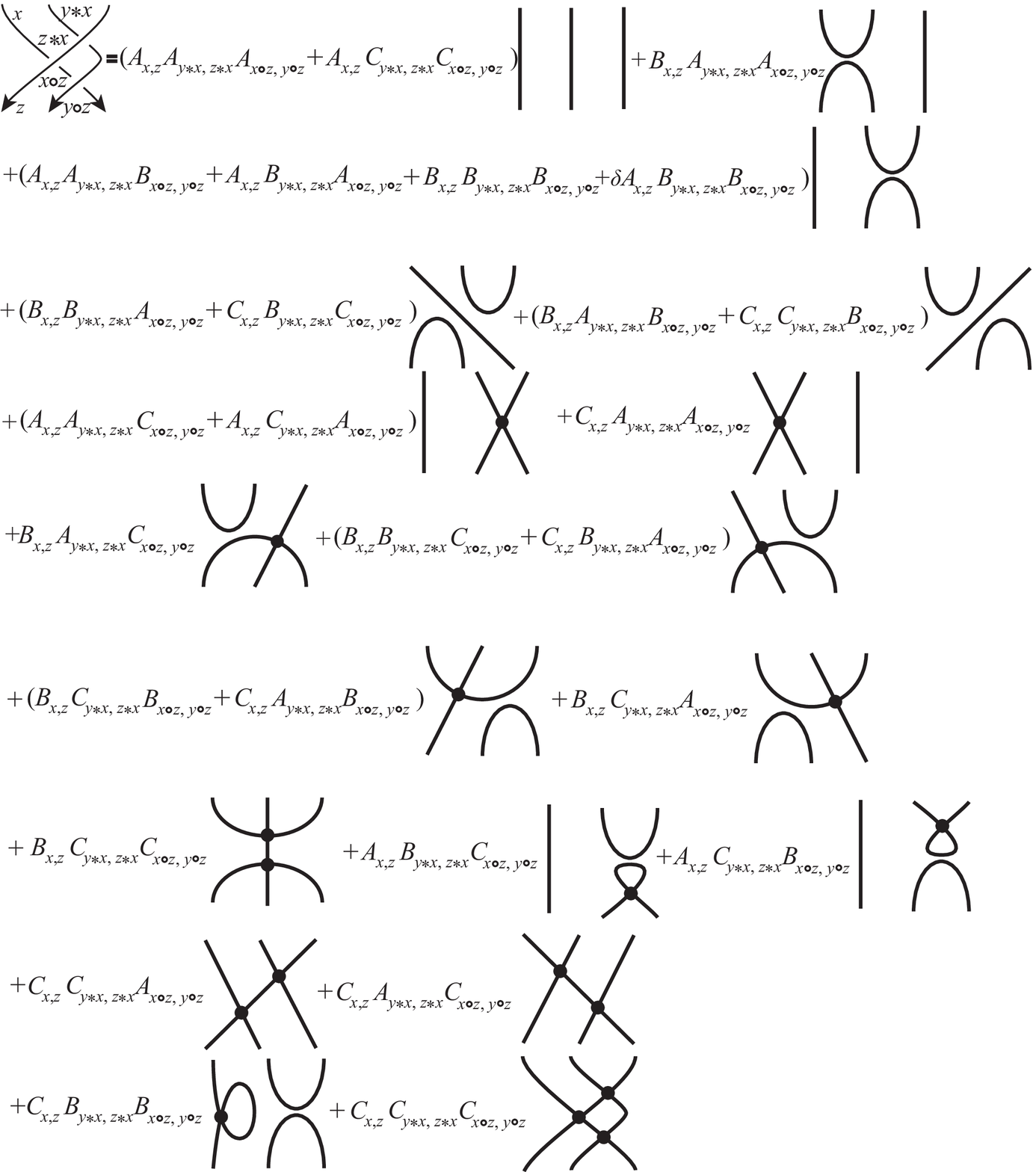}
  \caption{Relations from $\Omega{3}$}\label{bbr32}
 \end{figure}

As a result, we get the following definitions (cf.~\cite{NOR}).

Let $L$ be an oriented (virtual) link diagram with $n$ crossings and let
 $$
\mathcal{B}(L)=\langle x_1,\dots,x_{2n}\,|\,\mathrm{CR},\,\mathrm{R1},\,\mathrm{R2},\,\mathrm{R3},\,\mathrm{R4}\rangle.
 $$
be its fundamental biquandle. There are $3^n$ states of $L$, i.e.\ at each crossing we have either the positive smoothing, or negative smoothing, or the graphical vertex  (we disregard virtual crossings). For each state we have the product of $n$ factors of $A_{x,y}$, or $B_{x,y}$, or $C_{x,y}$, or $D_{x,y}$, or $E_{x,y}$, or $F_{x,y}$ times a $4$-valent graph with a cross structure.

 \begin{definition}
The {\em fundamental parity-biquandle bracket value} for $L$ is the sum of the contributions times $w^{-\mathrm{wr}(L)}$, where $\mathrm{wr}(L)$ is the writhe number of $L$.
 \end{definition}

 \begin{definition}\label{def:pbbr}
Let $X$ be a finite biquandle, $R$ be a commutative ring with unity. Let $f\in\mathrm{Hom}(\mathcal{B}(L),X)$ be an $X$-coloring and a sequence of maps $\beta=(A,B,C,D,E,F)$, $A,\,B,\,C,\,D,\,E,\,F\colon X\times X\to R$, with distinguished elements $\delta\in R$ and $w\in R^\prime$ satisfy the following relations:
 \begin{enumerate}
  \item[(i)]
$\delta A_{x,x}+B_{x,x}=w$, $\delta D_{x,x}+E_{x,x}=w^{-1}$ and $C_{x,x}=F_{x,x}=0$ for $\forall\,x\in X$,
  \item[(ii)]
$A_{x,y}F_{x,y}=C_{x,y}D_{x,y}=B_{x,y}F_{x,y}=C_{x,y}E_{x,y}=0$, $A_{x,y}D_{x,y}=B_{x,y}E_{x,y}=1-C_{x,y}F_{x,y}$ and
$\delta A_{x,y}D_{x,y}=-A_{x,y}E_{x,y}-B_{x,y}D_{x,y}$ for $\forall\,x,y\in X$,
  \item[(iii)]
 \begin{gather*}
A_{x,y}A_{y,z}A_{x\circ y,z\ast y}+C_{x,y}C_{y,z}A_{x\circ y,z\ast y}=A_{x,z}A_{y\ast x,z\ast x}A_{x\circ z,y\circ z}+A_{x,z}C_{y\ast x,z\ast x}C_{x\circ z,y\circ z},\\
A_{x,y}B_{y,z}B_{x\circ y,z\ast y}+C_{x,y}B_{y,z}C_{x\circ y,z\ast y}=B_{x,z}B_{y\ast x,z\ast x}A_{x\circ z,y\circ z}+C_{x,z}B_{y\ast x,z\ast x}C_{x\circ z,y\circ z},\\
B_{x,y}A_{y,z}B_{x\circ y,z\ast y}+B_{x,y}C_{y,z}C_{x\circ y,z\ast y}=B_{x,z}A_{y\ast x,z\ast x}B_{x\circ z,y\circ z}+C_{x,z}C_{y\ast x,z\ast x}B_{x\circ z,y\circ z},\\
A_{x,y}C_{y,z}A_{x\circ y,z\ast y}+C_{x,y}A_{y,z}A_{x\circ y,z\ast y}=C_{x,z}A_{y\ast x,z\ast x}A_{x\circ z,y\circ z},\\
A_{x,y}A_{y,z}C_{x\circ y,z\ast y}=A_{x,z}A_{y\ast x,z\ast x}C_{x\circ z,y\circ z}+A_{x,z}C_{y\ast x,z\ast x}A_{x\circ z,y\circ z},\\
A_{x,y}C_{y,z}B_{x\circ y,z\ast y}=B_{x,z}B_{y\ast x,z\ast x}C_{x\circ z,y\circ z}+C_{x,z}B_{y\ast x,z\ast x}A_{x\circ z,y\circ z},\\
B_{x,y}C_{y,z}B_{x\circ y,z\ast y}+B_{x,z}A_{y,z}C_{x\circ y,z\ast y}=B_{x,z}A_{y\ast x,z\ast x}C_{x\circ z,y\circ z},\\
A_{x,y}B_{y,z}C_{x\circ y,z\ast y}+C_{x,z}B_{y,z}B_{x\circ y,z\ast y}=B_{x,z}C_{y\ast x,z\ast x}A_{x\circ z,y\circ z},\\
C_{x,y}A_{y,z}B_{x\circ y,z\ast y}=B_{x,z}C_{y\ast x,z\ast x}B_{x\circ z,y\circ z}+C_{x,z}A_{y\ast x,z\ast x}B_{x\circ z,y\circ z},\\
C_{x,y}C_{y,z}B_{x\circ y,z\ast y}=B_{x,z}C_{y\ast x,z\ast x}C_{x\circ z,y\circ z},\\
A_{x,y}C_{y,z}C_{x\circ y,z\ast y}=C_{x,z}C_{y\ast x,z\ast x}A_{x\circ z,y\circ z},\\
C_{x,y}A_{y,z}C_{x\circ y,z\ast y}=C_{x,z}A_{y\ast x,z\ast x}C_{x\circ z,y\circ z},\\
B_{x,y}C_{y,z}A_{x\circ y,z\ast y}=C_{x,y}B_{y,z}A_{x\circ y,z\ast y}=B_{x,y}B_{y,z}C_{x\circ y,z\ast y}=C_{x,y}C_{y,z}C_{x\circ y,z\ast y}=0,\\
A_{x,z}B_{y\ast x,z\ast x}C_{x\circ z,y\circ z}=A_{x,z}C_{y\ast x,z\ast x}B_{x\circ z,y\circ z}\\
=C_{x,z}B_{y\ast x,z\ast x}B_{x\circ z,y\circ z}=C_{x,z}C_{y\ast x,z\ast x}C_{x\circ z,y\circ z}=0,\\
A_{x,y}A_{y,z}B_{x\circ y,z\ast y}=A_{x,z}B_{y\ast x,z\ast x}A_{x\circ z,y\circ z}+
A_{x,z}A_{y\ast x,z\ast x}B_{x\circ z,y\circ z}\\
+\delta A_{x,z}B_{y\ast x,z\ast x}B_{x\circ z,y\circ z}+B_{x,z}B_{y\ast x,z\ast x}B_{x\circ z,y\circ z},\\
B_{x,z}A_{y\ast x,z\ast x}A_{x\circ z,y\circ z}=B_{x,y}A_{y,z}A_{x\circ y,z\ast y}+A_{x,y}B_{y,z}A_{x\circ y,z\ast y}\\
+\delta B_{x,y}B_{y,z}A_{x\circ y,z\ast y}+B_{x,y}B_{y,z}B_{x\circ y,z\ast y}
   \end{gather*}
for $\forall\,x,y,z\in X$.
 \end{enumerate}
Here $A_{x,y}=A(x,y)$, $B_{x,y}=B(x,y)$, $C_{x,y}=C(x,y)$, $D_{x,y}=D(x,y)$, $E_{x,y}=E(x,y)$, $F_{x,y}=F(x,y)$.

We set the value of the fundamental parity-biquandle bracket value in $f$ for $\beta$ to be $\beta(f)_{gr}\in R\mathfrak{G}_\delta$.

The {\em parity-biquandle bracket multiset} of $L$ is the following multiset:
 $$
\Phi^{\beta,M}_{X,gr}(L)=\{\beta(f)_{gr}\,|f\in\mathrm{Hom}(\mathcal{B}(L),X)\}.
 $$
 \end{definition}

 \begin{theorem}
The parity-biquandle bracket multiset is an invariant of virtual links. Namely{\em,} if two virtual link diagrams $L_1$ and $L_2$ represent the same link then $\Phi^{\beta,M}_{X,gr}(L_1)$ is isomorphic to $\Phi^{\beta,M}_{X,gr}(L_2)$ for any biquandle $X$ and $\beta$ from Def.~{\em\ref{sec:parbiqbr}.\ref{def:pbbr}}.
 \end{theorem}

 \begin{proof}
Let $L_1$ is obtained from $L_2$ by a Reidemeister move. It follows from Def.~\ref{sec:kn&biq&par}.\ref{def:biq} that the fundamental biquandles $\mathcal{B}(L_1)$ and $\mathcal{B}(L_2)$ are isomorphic and to each biquandle coloring $f_1\in \mathrm{Hom}(\mathcal{B}(L_1),X)$ of $L_1$ it is associated the biquandle coloring $f_2\in \mathrm{Hom}(\mathcal{B}(L_2),X)$ of $L_2$. Then using relations from Def.~\ref{sec:parbiqbr}.\ref{def:pbbr} we get that $\beta(f_1)_{gr}=\beta(f_2)_{gr}$.
 \end{proof}

 \begin{example}
Let $\beta=(A,B,0,A^{-1},B^{-1},0),$ $A,\,B\in R^\prime,$ in Def.~\ref{sec:parbiqbr}.\ref{def:pbbr} then $\Phi^{\beta,M}_X(L)$ is isomorphic to $\Phi^{\beta,M}_{X,gr}(L)\Big/\bigcirc=\delta$, where $\bigcirc$ is a closed curve which may contain only virtual crossings.
 \end{example}

 \begin{example}
Let $K$ be a virtual knot, $X$ be the biquandle from Example~\ref{sec:kn&biq&par}.\ref{ex:parbiq}, $f\in\mathrm{Hom}(\mathcal{B}(K),X)$ and $\beta=(A,B,C,D,E,F)$, where
 \begin{align*}
A_{x,x+1}&=B_{x,x+1}=D_{x,x+1}=E_{x,x+1}=C_{x,x+1}+1=F_{x,x+1}+1\\
&=A_{x,x}+1=B_{x,x}+1=D_{x,x}+1=E_{x,x}+1=C_{x,x}=F_{x,x}=0,
 \end{align*}
$x\in\mathbb{Z}_2$, $\delta=0$, $w=1$, then $\Phi^{\beta,M}_{X,gr}(L)$ is isomorphic to the multiset containing the bracket $[K]$ twice and zero elements.
 \end{example}

 \section*{Acknowledgments}

The authors are grateful to I.\,M.~Nikonov and S.~Kim for his interest to the work and S.~Nelson for his talk in our seminar, which stimulates us to write this paper.

 \end{document}